\pdfoutput=1
\documentclass[english]{jnsao}
\usepackage[utf8]{inputenc}
\usepackage{csquotes}
\usepackage{babel}
\usepackage{amsmath,amsthm,amssymb}
\usepackage[nameinlink,capitalise]{cleveref}
\usepackage{multirow}

\usepackage{etoolbox}
\patchcmd{\thebibliography}
  {\settowidth}
  {\setlength{\parsep}{0.2pt}\setlength{\itemsep}{0pt plus 0.5pt}\settowidth}
  {}{}
\apptocmd{\thebibliography}
  {\small}
  {}{}

\newtheorem{assumption}[theorem]{Assumption}

\newcommand\norm[2]{\left\Vert#1\right\Vert_{#2}}

\newcommand\N{\mathbb{N}}
\newcommand\R{\mathbb{R}}

\DeclareMathOperator*{\argmin}{\operatorname{argmin}}
\DeclareMathOperator*{\argmax}{\operatorname{argmax}}
\DeclareMathOperator{\dom}{dom}
\DeclareMathOperator{\intr}{int}
\DeclareMathOperator{\conv}{conv}
\DeclareMathOperator{\gph}{gph}
\DeclareMathOperator{\cl}{cl}
\DeclareMathOperator{\epi}{epi}

\DeclareMathAlphabet{\mathpzc}{OT1}{pzc}{m}{it}
\newcommand\oo{\mathpzc{o}}

\def\thetitle{Sufficient optimality conditions in bilevel programming}

\def\theauthor{Patrick Mehlitz, Alain B.\ Zemkoho}

\def\theauthorA{Patrick Mehlitz} 
\def\theauthorB{Alain B.\ Zemkoho}

\author{%
\theauthorA\thanks{
	Institute of Mathematics, Chair of Optimal Control.
	Brandenburgische Technische Universität Cottbus--Senftenberg, Germany.
	\email{mehlitz@b-tu.de}
	},\,
\theauthorB\thanks{
	School of Mathematics.
	University of Southampton, United Kingdom,
	\email{a.b.zemkoho@soton.ac.uk}
	}
}

\title{\thetitle}
\shorttitle{Sufficient optimality conditions in bilevel programming}
\shortauthor{\theauthor}
\date{\ISOToday}

\AtBeginDocument{
    \hypersetup{
        pdftitle = {\thetitle},
        pdfauthor = {\theauthor},
    }
}

\begin{document}

\maketitle

\begin{abstract}
	This paper is concerned with the derivation of first- and second-order sufficient optimality conditions
	for optimistic bilevel optimization problems involving smooth functions.
	First-order sufficient optimality conditions are obtained by estimating the tangent cone to the
	feasible set of the bilevel program in terms of initial problem data. This is done by exploiting several different
	reformulations of the hierarchical model as a single-level problem.
	To obtain second-order sufficient optimality conditions, we exploit the so-called value
	function reformulation of the bilevel optimization problem, which is then tackled with the aid of second-order
	directional derivatives. The resulting conditions can be stated in terms of initial problem data in several
	interesting situations comprising the settings where the lower level is linear or possesses strongly stable
	solutions.
	\\[2ex]
\\[2ex]
\end{abstract}

\section{Introduction}\label{Introduction}
Bilevel programming is without any doubt one of the most popular research areas in mathematical
optimization. On the one hand, this interest is a consequence of numerous underlying real-world
applications from a wide range of areas, including economics, finance, logistics, and chemistry, 
which can be modeled as hierarchical
programs with two decision makers. On the other hand, bilevel programs are mathematically
challenging since they are inherently nonsmooth, nonconvex, and irregular in the sense that
a reformulation of the hierarchical model as single-level program results in surrogate problems
which suffer from an inherent lack of smoothness, convexity, and regularity. A detailed introduction
to the topic of bilevel programming can be found in the monographs
\cite{Bard1998,Dempe2002,DempeKalashnikovPerezValdesKalashnykova2015,ShimizuIshizukaBard1997}.

Despite the vast amount of work done so far on bilevel optimization, very little appears to have been 
done on sufficient optimality conditions for this problem class. 
The aim of this paper is to make new contributions in this area. 
Note that as in standard nonlinear optimization, such results have the potential to help accelerate work 
on numerical methods and stability analysis of bilevel programs. 
To proceed, we consider the standard bilevel optimization problem
\begin{equation}\label{eq:BPP}\tag{BPP}
  	\min\limits_{x,y} \{F(x,y)\,|\,G(x,y)\leq 0,\; y\in S(x)\},
\end{equation}
also known as the upper level problem. 
Here, the set-valued mapping $S \colon \mathbb{R}^n \rightrightarrows \mathbb{R}^m$ represents the
optimal solution mapping of the so-called lower level problem and is defined by
\begin{equation}\label{eq:lower-level}
	\forall x\in\R^n\colon\quad S(x):= \argmin\limits_y\{f(x,y)\,|\,g(x,y)\leq 0\}.
\end{equation}
The functions $F, f \colon\R^n\times\R^m\to\R$, $G\colon\R^n\times\R^m\to\R^p$,  and $g \colon\R^n\times\R^m\to\R^q$
are assumed to be twice continuously differentiable. Throughout the manuscript, the component maps of $G$ and $g$ will be denoted by
$G_1,\ldots,G_p\colon\R^n\times\R^m\to\R$ and $g_1,\ldots,g_q\colon\R^n\times\R^m\to\R$, respectively.
Recall that problem \eqref{eq:BPP} has a two-level structure with $F$ (resp.\ $f$) denoting the upper level (resp.\ lower level)
objective function and $G$ (resp.\ $g$) representing the upper level (resp.\ lower level) constraint functions.
Observing that the minimization in \eqref{eq:BPP} is carried out w.r.t.\ $x$ and $y$, the given model
is closely related to the optimistic approach of the bilevel programming problem, see \cite[Section~5.3]{Zemkoho2012} for details.

During the last decades, many papers appeared where necessary optimality conditions for \eqref{eq:BPP}
are derived via different single-level reformulations of the model. 
However, there is a significant gap in the literature on \emph{sufficient} optimality conditions for \eqref{eq:BPP}. 
In \cite[Theorem~4.2]{Dempe1992}, the author presents a sufficient optimality condition for bilevel programming
problems with uniquely determined lower level optimal solutions. The underlying analysis is based on conditions which ensure
that the lower level optimal solution operator is singleton-valued and directionally differentiable.
As demonstrated in \cite[Theorem~6.1]{YeZhu1995}, exploiting standard second-order theory from smooth nonlinear programming,
it is possible to infer sufficient optimality conditions for bilevel programming problems where upper and lower level
problem are not constrained.
In \cite{DempeGadhi2010}, the authors derive second-order optimality conditions for bilevel optimization
problems using approximate derivatives. However, these optimality conditions are not stated in terms of
initial problem data and hence, for large parts, do not reflect the inherent difficulties of bilevel optimization.
A derivative-free sufficient optimality condition for bilevel optimization problems which is based on
a $\max$-$\min$-equality is presented in \cite{AboussororAdly2018}.
In \cite{GadhiHamdaouriElIdrissi2019}, the authors exploit generalized convexity assumptions and
a new reformulation of the bilevel programming problem based on a so-called $\Psi$-function in order to infer a
first-order sufficient optimality condition. Again, we emphasize that the obtained condition is not stated
in terms of initial problem data.

In \cite{BazaraaSheraliShetty1993,BenTal1980,McCormick1967}, second-order sufficient optimality
conditions for standard nonlinear programs have been derived which are based on the notion of
a suitable critical cone and the second-order derivative of the underlying Lagrangian.
Supposing that the feasible set of the program under consideration is given in an abstract way
(e.g.\ by generalized inequalities), a generalization of these results is possible using
the theory of suitable second-order tangent sets,
see \cite{BonnansCominettiShapiro1999,BonnansShapiro2000,CambiniMarteinVlach1999,Mehlitz2019,MohammadiMordukhovichSarabi2019,Penot1998}.
Let us mention that using generalized second-order derivatives, it is possible to derive second-order
sufficient conditions in situations where nonsmooth optimization problems are under consideration,
see \cite{BenTalZowe1982,BonnansShapiro2000,RockafellarWets1998,RueckmannShapiro2001,Studniarski1986}.

In this paper, we are going to derive new first- and second-order sufficient optimality conditions for \eqref{eq:BPP}
in terms of initial problem data. For the first-order sufficient conditions, we exploit several approaches
in order to approximate the tangent cone to the feasible set of \eqref{eq:BPP} from above.
The resulting primal conditions generally require that a certain system of equations does not
possess a nontrivial solution. In order to obtain second-order sufficient optimality conditions
for \eqref{eq:BPP}, we suggest to exploit the so-called value function reformulation of \eqref{eq:BPP},
see \cref{sec:preliminaries_bilevel}, as well as second-order directional derivatives of all
the involved functions. Our analysis is based on related investigations for standard nonlinear
problems, see, e.g., \cite{BenTalZowe1982,Shapiro1988,Shapiro1988b}, and semi-infinite nonlinear
problems, see \cite{RueckmannShapiro2001}.
Throughout the paper, we present several examples in order to visualize the applicability of the obtained theory.

The remaining parts of this manuscript are organized as follows:
In \cref{sec:preliminaries}, we first introduce the basic notation used throughout the paper.
Afterwards, some fundamentals of variational analysis are recalled. Furthermore, we state all
notions of first- and second-order directional differentiability that we are going to exploit in
this paper. Finally, we briefly summarize some essentials of constrained programming.
\Cref{sec:preliminaries_bilevel} recalls some fundamental concepts from bilevel optimization.
In \cref{sec:first_order_sufficient_conditions}, first-order sufficient
optimality conditions for \eqref{eq:BPP} are derived. Therefore, we estimate the tangent cone to the
feasible set of \eqref{eq:BPP} from above using different single-level surrogate problems associated with
\eqref{eq:BPP}, namely the value function reformulation, the generalized equation reformulation, and
the Karush--Kuhn--Tucker (KKT) reformulation.
In \cref{sec:second_order_optimality_conditions}, second-order sufficient optimality conditions
for \eqref{eq:BPP} will be derived. Therefore, we exploit the value function reformulation.
In order to describe the curvature of the optimal value function associated with \eqref{eq:lower-level}, we make
use of its second-order directional derivative. We first obtain an abstract result comprising some
generalized derivatives of the optimal value function. In three exemplary settings, namely where
the parametric optimization problem in \eqref{eq:lower-level} is fully linear, linear in $y$ and the
parameters $x$ enters the lower level objective linearly, and where the lower level optimal solutions are strongly
stable, the result is specified in terms of initial problem data.
Finally, we close the paper in \cref{sec:conclusions} with some concluding remarks comprising
directions for future research.

\section{Notation and preliminaries}\label{sec:preliminaries}

\subsection{Basic notation}

For a matrix $A\in\R^{m\times n}$ and a set $C\subseteq\R^n$, let us set $AC:=\{Ax\,|\,x\in C\}$.
Furthermore, we define by $\operatorname{dist}(x, C):=\inf\{\norm{x-y}{2}\,|\,y\in C\}$ the
distance between $x$ and $C$. Here, $\norm{\cdot}{2}$ represents the Euclidean norm in $\R^n$.
Forthwith, $\conv C$, $\cl C$, and $\intr C$ denote the convex hull, the closure,
and the interior of $C$, respectively.
We use $I_n\in\R^{n\times n}$ to express the identity.
By $\dom \theta:=\{x\in\R^n\,|\,|\theta(x)|<\infty\}$ and
$\epi\theta:=\{(x,\alpha)\in\R^n\times\R\,|\,\theta(x)\leq\alpha\}$, we denote the domain and
the epigraph of a function $\theta\colon\R^n\to\overline\R$, respectively,
where $\overline\R$ represents the extended real line.
Similarly, for a set-valued mapping $\Theta\colon\R^n\rightrightarrows\R^m$, we use
$\dom\Theta:=\{x\in\R^n\,|\,\Theta(x)\neq\varnothing\}$ and
$\gph\Theta:=\{(x,y)\in\R^n\times\R^m\,|\,y\in\Theta(x)\}$ in order to represent the domain
and the graph of $\Theta$, respectively.
At $\bar x\in\dom\Theta$, $\Theta$ is said to be locally bounded whenever there exist a neighborhood
$U\subseteq\R^n$ of $\bar x$ and a bounded set $B\subseteq\R^m$ such that $\Theta(x)\subseteq B$ 
holds for all $x\in U$.
Furthermore, we call $\Theta$ metrically subregular at $(\bar x,\bar y)\in\gph\Theta$ whenever there are
a neighborhood $U\subseteq\R^n$ of $\bar x$ and a constant $\kappa>0$ such that
\[
	\forall x\in U\colon\quad
	\operatorname{dist}(x,\Theta^{-1}(\bar y))
	\leq
	\kappa\operatorname{dist}(\bar y,\Theta(x))
\]
is valid. Here, we used $\Theta^{-1}(\bar y):=\{x\in\R^n\,|\,\bar y\in\Theta(x)\}$ for the
preimage of $\bar y$ under $\Theta$.

\subsection{Variational Analysis}

Let $C\subseteq\R^n$ be a nonempty set. Then, we call
\[
	C^\circ
	:=
	\{y\in\R^n\,|\,\forall x\in C\colon\,x^\top y\leq 0\}
\]
the polar cone of $C$. Note that $C^\circ$ is a nonempty, closed, convex cone.
For a nonempty index set $I$ and a family $\{C_i\}_{i\in I}$ of sets from $\R^n$, one easily
obtains the polarization rule
\[
	\left(\bigcup\nolimits_{i\in I}C_i\right)^\circ
	=
	\bigcap\nolimits_{i\in I}C^\circ_i
\]
by definition of the polar cone. For the polyhedral cone
\[
	P
	:=
	\left\{
		x\in\R^n
		\,\middle|\,
		Ax\leq 0,\,Bx=0
	\right\},
\]
where $A\in\R^{m_1\times n}$ and $B\in\R^{m_2\times n}$ are arbitrary matrices, one has (e.g., by Farkas' lemma) that
\[
	P^\circ
	=
	\left\{
		A^\top \mu+B^\top\nu\in\R^n
		\,|\,
		\mu\in\R^{m_1},\,\mu\geq 0,\,\nu\in\R^{m_2}
	\right\}.
\]
For a some point $\bar x\in C$, we refer to
\[
	\mathcal T_C(\bar x):=
		\left\{
			d\in\R^n\,\middle|\,
				\exists\{t_k\}_{k\in\N}\subseteq\R\,\exists\{d_k\}_{k\in\N}\subseteq\R^n\colon\,
				t_k\downarrow 0,\,d_k\to d,\,\bar x+t_kd_k\in C\;\forall k\in\N
		\right\}
\]
as the tangent (or Bouligand) cone to $C$ at $\bar x$. Furthermore,
\[
	\widehat{\mathcal N}_C(\bar x):=\mathcal T_C(\bar x)^\circ
\]
is referred to as the regular (or Fr\'{e}chet) normal cone to $C$ at $\bar x$.
The following lemma, which presents a calculus rule for tangents to linear images,
is taken from \cite[Proposition~4.3.9]{AubinFrankowska2009}.
\begin{lemma}\label{lem:direct_image}
	Fix $C\subseteq\R^n$ and a matrix $A\in\R^{m\times n}$ as well as some
	point $\bar x\in C$ with
	\[
		\operatorname{ker} A\cap\mathcal T_C(\bar x)=\{0\}.
	\]
	Then, it holds $\mathcal T_{AC}(A\bar x)=\cl(A\mathcal T_C(\bar x))$.
	Above, $\operatorname{ker} A$ denotes the null space of $A$.
\end{lemma}

\subsection{Generalized differentiation}\label{sec:generalized_differentiation}

Throughout the section, we fix a function $\theta\colon\R^n\to\overline\R$ as well as some
point $\bar x\in\dom \theta$.
Let us first recall some fundamental notions of (first-order) directional differentiability,
see, e.g., \cite{Clarke1983,Shapiro1990}.
For a direction $d\in\R^n$, the limits
\[
	\theta^+(\bar x;d):=\limsup\limits_{t\downarrow 0}\frac{\theta(\bar x+td)-\theta(\bar x)}{t}
	\qquad\mbox{ and } \qquad
	\theta^-(\bar x;d):=\liminf\limits_{t\downarrow 0}\frac{\theta(\bar x+td)-\theta(\bar x)}{t}
\]
are, respectively, referred to as the upper and lower Dini directional derivative of $\theta$ at $\bar x$ in
direction $d$. In case of existence, we call
\[
	\theta'(\bar x;d):=\lim\limits_{t\downarrow 0}\frac{\theta(\bar x+td)-\theta(\bar x)}{t}
\]
the directional derivative (in G\^{a}teaux's sense) of $\theta$ at $\bar x$ in direction $d$.
Similarly, we introduce the Hadamard and Clarke directional derivative of $\theta$ at $\bar x$ in
direction $d$, respectively, as
\[
	\theta^\ast(\bar x;d):=\lim\limits_{t\downarrow 0,\,d'\to d}\frac{\theta(\bar x+td')-\theta(\bar x)}{t}
	\qquad\mbox{ and } \qquad
	\theta^\circ(\bar x;d):=\limsup\limits_{t\downarrow 0,\,x\to\bar x}\frac{\theta(x+td)-\theta(x)}{t}.
\]
We call $\theta$ directionally differentiable (resp.\ Hadamard directionally differentiable, Clarke directionally
differentiable) at $\bar x$ if the limit $\theta'(\bar x;d)$ (resp.\ $\theta^\ast(\bar x;d)$, $\theta^\circ(\bar x;d)$)
exists for each $d\in\R^n$. Clearly, if $\theta$ is continuously differentiable at $\bar x$, then all
these generalized derivatives coincide with $\nabla \theta(\bar x)^\top d$ for each $d\in\R^n$.
Let us note that whenever $\theta$ is locally Lipschitz continuous at $\bar x$, then it is Clarke
directionally differentiable there. If, in addition to local Lipschitz continuity, $\theta$ is
directionally differentiable at $\bar x$, then it is Hadamard directionally differentiable at $\bar x$
and $\theta'(\bar x;d)$ as well as $\theta^\ast(\bar x;d)$ coincide.
We call $\theta$ Clarke-regular at $\bar x$ if $\theta'(\bar x;d)=\theta^\circ(\bar x;d)$
holds for all $d\in\R^n$. One can check that convex functions are Clarke regular at all points from $\intr\dom\theta$.
Clearly, continuously differentiable functions are Clarke regular as well.

Supposing that $\theta$ is locally Lipschitz continuous at $\bar x$, its Clarke subdifferential at $\bar x$ given by
\[
	\partial^c\theta(\bar x):=\{\xi\in\R^n\,|\,\forall d\in\R^n\colon\;\xi^\top d\leq \theta^\circ(\bar x;d)\}
\]
is nonempty, convex, and compact. Let us note that this subdifferential construction enjoys full calculus, see
\cite{Clarke1983}, but is comparatively large w.r.t.\ set inclusion, i.e., it yields very weak necessary
optimality conditions. Using the subdifferential, Clarke's directional derivative can be recovered by the formula
\[
	\forall d\in\R^n\colon\quad
	\theta^\circ(\bar x;d)=\max\{\xi^\top d\,|\,\xi\in\partial^c\theta(\bar x)\}.
\]
Next, let $\theta$ be a convex function. Then, we call
\[
	\partial\theta(\bar x):=\{\xi\in\R^n\,|\,\forall x\in\R^n\colon\,\theta(x)\geq\theta(\bar x)+\xi^\top(x-\bar x)\}
\]
the subdifferential of $\theta$ at $\bar x$ which is nonempty, closed, and convex.
It is well known that
\[
	\partial\theta(\bar x)=\{\xi\in\R^n\,|\,\forall d\in\R^n\colon\,\xi^\top d\leq \theta'(\bar x;d)\}
\]
holds true. Thus, whenever a point $\bar x$ is under consideration where $\theta$ is locally Lipschitz continuous,
then $\partial^c\theta(\bar x)=\partial\theta(\bar x)$ follows by Clarke regularity of $\theta$ at $\bar x$.

Below, we study a second-order directional derivative which has been exploited for the
second-order analysis of optimization problems in e.g.\
\cite{BenTalZowe1982,BonnansCominettiShapiro1999,BonnansShapiro2000,RueckmannShapiro2001,Shapiro1988,Shapiro1988b}.
For further information and some illustrative examples, we refer the interested reader
to the aforementioned references.
We say that $\theta$ is
second-order directionally differentiable at $\bar x$ if the limit
\[
	\theta''(\bar x;d,w)
	:=
	\lim\limits_{t\downarrow 0}
	\frac{\theta(\bar x+td+\tfrac12t^2w)-\theta(\bar x)-t\theta'(\bar x;d)}{\tfrac12 t^2}
\]
exists for each choice of $d,w\in\R^n$. In case of existence, the above limit is referred to
as second-order directional derivative of $\theta$ at $\bar x$ w.r.t.\ the directions $d$ and $w$.
Supposing that $\theta''(\bar x;d,w)$ is finite, 
it holds $(\alpha\theta)''(\bar x;d,w)=\alpha\,\theta''(\bar x;d,w)$ for each $\alpha\in\R$.
Furthermore, the sum rule is available for the second-order directional derivative provided all
involved functions are second-order directionally differentiable.
If $\theta$ is twice continuously differentiable at $\bar x$, then a second-order Taylor expansion yields
\begin{equation}\label{eq:second_order_taylor_expansion}
	\theta(\bar x+td+\tfrac12 t^2w)
	=
	\theta(\bar x)+\nabla\theta(\bar x)^\top d
	+
	\tfrac12t^2\left(\nabla \theta(\bar x)^\top w+d^\top\nabla^2\theta(\bar x)d\right)
	+
	\oo(t^2)
\end{equation}
which shows that for all $d,w\in\R^n$, one has
\begin{equation}\label{eq:second_order_directional_derivative_smooth_function}
	\theta''(\bar x;d,w)=\nabla\theta(\bar x)^\top w+d^\top\nabla^2\theta(\bar x)d.
\end{equation}

Subsequently, we state a calculus rule for the computation of second-order directional
derivatives associated with the pointwise maximum of finitely many second-order directionally
differentiable functions.
It follows from the chain rule, see \cite[Proposition~2.53]{BonnansShapiro2000}, and can
be found in the literature, see \cite[Lemma~3.2]{BenTalZowe1982} or \cite[Section~4.1]{BonnansCominettiShapiro1999}.
\begin{lemma}\label{lem:max_rule}
	Let $\theta_1,\ldots,\theta_p\colon\R^n\to\overline\R$ be Lipschitz continuous and
	second-order directionally differentiable at $\bar x\in\bigcap_{i=1}^p\dom\theta_i$.
	Furthermore, set $\theta:=\max\{\theta_1,\ldots,\theta_p\}$.
	Then, $\theta$ is second-order directionally differentiable at $\bar x$ and it holds
	\[
		\forall d\in\R^n\colon\quad
		\theta'(\bar x;d)=\max\{\theta'_i(\bar x;d)\,|\,i\in I(\bar x)\}
	\]
	as well as
	\[
		\forall d,w\in\R^n\colon\quad
		\theta''(\bar x;d,w)=\max\{\theta''_i(\bar x;d,w)\,|\,i\in I(\bar x,d)\},
	\]
	where we use
	\[
			I(\bar x) :=\{i\in\{1,\ldots,p\}\,|\,\theta(\bar x)=\theta_i(\bar x)\}\qquad \mbox{ and }\qquad
			I(\bar x,d):=\{i\in I(\bar x)\,|\,\theta'(\bar x;d)=\theta'_i(\bar x;d)\}.
	\]
\end{lemma}

In order to study sufficient optimality conditions with the aid of the second-order directional
derivative introduced above, the presence of an additional regularity condition is indispensable.
Here, we rely on the concept of second-order epi-regularity which dates back to
\cite{BonnansCominettiShapiro1999} and is studied in \cite[Sections~3.3.4 and 3.3.5]{BonnansShapiro2000}.
\begin{definition}\label{def:second_order_epiregularity}
	Let $\theta\colon\R^n\to\overline\R$ be a given function and fix $\bar x\in\operatorname{dom}\theta$
	where $\theta$ is second-order directionally differentiable.
	Then, $\theta$ is said to be second-order epi-regular at $\bar x$ in direction $d\in\R^n$ if
	for each path $w\colon\R_+\to\R^n$ which satisfies $t w(t)\to 0$ as $t\downarrow 0$, we have
	\[
		\theta(\bar x+td+\tfrac12t^2 w(t))
		\geq
		\theta(\bar x)+t\,\theta'(\bar x;d)+\tfrac12t^2\theta''(\bar x;d,w(t))+\oo(t^2).		
	\]
	We say that $\theta$ is second-order epi-regular at $\bar x$ if it is second-order epi-regular at $\bar x$
	in each direction from $\R^n$.
\end{definition}

Due to \eqref{eq:second_order_taylor_expansion} and \eqref{eq:second_order_directional_derivative_smooth_function},
each twice continuously differentiable functions is second-order epi-regular at each point
in each direction.
By definition, the sum of two second-order epi-regular functions is second-order epi-regular as well.
Furthermore, we note that a function $\theta\colon\R^n\to\overline\R$, which is second-order directionally
differentiable and locally Lipschitz continuous at $\bar x\in\dom\theta$, is
second-order epi-regular at $\bar x$ in direction $d\in\R^n$ if and only if the set
$\epi\theta$ is outer second-order regular at $(\bar x,\theta(\bar x))$, see
\cite[Definition~3.85, Proposition~3.95]{BonnansShapiro2000}, and the latter particularly holds
whenever $\epi\theta$ is a polyhedron.
\begin{lemma}\label{lem:second_order_epiregularity_max}
	Let $\theta_1,\ldots,\theta_p\colon\R^n\to\overline{\R}$ be Lipschitz continuous, second-order directionally differentiable, and
	second-order epi-regular at $\bar x\in\bigcap_{i=1}^p\dom\theta_i$.
	Furthermore, set $\theta:=\max\{\theta_1,\ldots,\theta_p\}$.
	Then, $\theta$ is second-order epi-regular at $\bar x$ as well.
\end{lemma}
\begin{proof}
	Fix an arbitrary direction $d\in\R^n$ and an arbitrary path $w\colon\R_+\to\R^n$ satisfying
	$t w(t)\to 0$ as $t\downarrow 0$. For each $t\in\R_+$, there is an active index $i_0(t)\in \{1,\ldots,p\}$
	which satisfies
		\[
		\theta(\bar x)=\theta_{i_0(t)}(\bar x),
		\qquad
		\theta'(\bar x;d)=\theta'_{i_0(t)}(\bar x;d),
		\qquad
		\theta''(\bar x;d,w(t))=\theta''_{i_0(t)}(\bar x;d,w(t)),
	\]
	see \cref{lem:max_rule}.
	Noting that all the functions $\theta_i$ are second-order epi-regular at $\bar x$ in direction $d$, it holds
	\[
		\theta_i(\bar x+td+\tfrac12t^2w(t))\geq\theta_i(\bar x)+t\theta_i'(\bar x;d)+\tfrac12t^2\theta_i''(\bar x;d,w(t))+\oo_i(t^2)
	\]
	for all $i=1,\ldots,p$.
	For each $t\in\R_+$, we set $\oo(t):=\min\{\oo_i(t)\,|\,i\in\{1,\ldots,p\}\}$.
	Then, we have
	\begin{align*}
		\theta(\bar x+td+\tfrac12 t^2w(t))
		&\geq
		\theta_{i_0(t)}(\bar x+td+\tfrac12 t^2w(t))\\
		&\geq
		\theta_{i_0(t)}(\bar x)+t\theta'_{i_0(t)}(\bar x;d)+\tfrac12 t^2\theta''_{i_0(t)}(\bar x;d,w(t))+\oo(t^2)\\
		&=\theta(\bar x)+t\theta'(\bar x;d)+\tfrac12 t^2\theta''(\bar x;d,w(t))+\oo(t^2).
	\end{align*}
	Since $d$ was chosen arbitrarily, the desired result follows.
\end{proof}

\subsection{Preliminaries on constrained programming}\label{sec:prelimiaries_nonlinear_programming}

In this section, we review some optimality conditions for the general optimization problem
\begin{equation}\label{eq:nonlinear_problem}\tag{P}
	\min\{\theta(x)\,|\,x\in X\},
\end{equation}
where $\theta\colon\R^n\to\overline\R$ is a given functional and $X\subseteq\R^n$ is a nonempty, closed set.
First of all, the following first-order necessary optimality condition can be easily proven using standard
arguments.
\begin{lemma}\label{lem:abstract_first_order_necessary_optimality_condition}
	Let $\bar x\in\dom\theta$ be a local minimizer of \eqref{eq:nonlinear_problem}, where $\theta$ is
	Hadamard directionally differentiable.
	Then, it holds that
	\[
		\forall d\in\mathcal T_X(\bar x)\colon\quad
		\theta^\ast(\bar x;d)\geq 0.
	\]
\end{lemma}

Next, we want to deal with sufficient optimality conditions which address \eqref{eq:nonlinear_problem}.
To proceed, we use the notion of growth conditions.
\begin{definition}\label{def:growth_conditions}
	Fix $\bar x\in X$. We say that the growth condition of order $\alpha>0$ holds
	at $\bar x$ for \eqref{eq:nonlinear_problem} whenever there are a neighborhood $U\subseteq\R^n$ of $\bar x$
	and a constant $C>0$ such that the following is true:
	\[
		\forall x\in X\cap U\colon\quad
		\theta(x)\geq \theta(\bar x)+C\norm{x-\bar x}{2}^\alpha.
	\]
\end{definition}

Clearly, whenever the growth condition of order $\alpha>0$ holds at some point $\bar x\in X$ for \eqref{eq:nonlinear_problem},
then $\bar x$ is a strict local minimizer of \eqref{eq:nonlinear_problem}.
A sufficient optimality condition for \eqref{eq:nonlinear_problem} is said to be of order $\alpha>0$,
if it implies validity of the growth condition of order $\alpha$ for \eqref{eq:nonlinear_problem} at the reference
point. Sufficient optimality conditions from the literature are of order $1$ or $2$ in most of the cases.

Below, we first study a rather general first-order sufficient optimality condition for \eqref{eq:nonlinear_problem}.
The proof, although it is folklore, is included for the reader's convenience.
\begin{proposition}\label{prop:abstract_first_order_sufficient_condition}
	Let $\bar x\in \dom\theta$ be a feasible point of \eqref{eq:nonlinear_problem},
	where $\theta$ is Hadamard directionally differentiable
	and assume that
	\[
		\forall d\in\mathcal T_X(\bar x)\setminus\{0\}\colon\quad
		\theta^\ast(\bar x;d)>0.
	\]
	Then, the first-order growth condition holds at $\bar x$ for \eqref{eq:nonlinear_problem}.
	Particularly, $\bar x$ is a strict local minimizer of \eqref{eq:nonlinear_problem}.
\end{proposition}
\begin{proof}
	Assume on the contrary that the first-order growth condition does not hold at $\bar x$ for \eqref{eq:nonlinear_problem}.
	Then, we find a sequence $\{x_k\}_{k\in\N}\subseteq X$ converging to $\bar x$ such that
	\[
		\forall k\in\N\colon\quad \theta(x_k)<\theta(\bar x)+\tfrac1k\norm{x_k-\bar x}{2}
	\]
	holds true. We set $t_k:=\norm{x_k-\bar x}{2}$ and $d_k:=(x_k-\bar x)/t_k$ for each $k\in\N$ and observe $t_k\downarrow 0$.
	Furthermore, $\{d_k\}_{k\in\N}$ is bounded and converges w.l.o.g.\ to some $d\in\mathcal T_X(\bar x)\setminus\{0\}$.
	By definition of Hadamard's directional derivative, we obtain
	\begin{align*}
		\theta^\ast(\bar x;d)
		=
		\lim\limits_{k\to\infty}\frac{\theta(\bar x+t_kd_k)-\theta(\bar x)}{t_k}
		=
		\lim\limits_{k\to\infty}\frac{\theta(x_k)-\theta(\bar x)}{t_k}
		\leq
		\lim\limits_{k\to\infty}\frac{\norm{x_k-\bar x}{2}}{k\,t_k}
		=
		0
	\end{align*}
	which contradicts the proposition's assumptions.
\end{proof}

Observe that the necessary and sufficient first-order optimality conditions from
\cref{lem:abstract_first_order_necessary_optimality_condition} and
\cref{prop:abstract_first_order_sufficient_condition}, respectively, only differ w.r.t.\ the
appearing relation the Hadamard directional derivative has to satisfy for nonvanishing tangent
directions to $X$ at the reference point. That is why we speak of a pair of \emph{no-gap} first-order
optimality conditions.

In numerous models of mathematical optimization, the feasible set $X$ is described via
inequality constraints. Thus, let us assume for a moment that $X$ is given by
\begin{equation}\label{eq:feasible_set_via_inequalities}
	X:=\{x\in\R^n\,|\,\forall i\in\{1,\ldots,m\}\colon\,\theta_i(x)\leq 0\}
\end{equation}
where $\theta_1,\ldots,\theta_m\colon\R^n\to\overline{\R}$ are given functionals.
In order to apply the sufficient optimality condition from \cref{prop:abstract_first_order_sufficient_condition}
efficiently to this particular setting, one needs a computable upper approximate of
the tangent cone.
It is a standard idea in mathematical programming to exploit a linearization of the constraining functionals
$\theta_1,\ldots,\theta_m$ for that purpose.
Let us fix a point $\bar x\in X$ where the functionals $\theta_1,\ldots,\theta_m$
are locally Lipschitz continuous. We call
\[
	\mathcal L_X(\bar x):=\{d\in\R^n\,|\,\theta^-_i(\bar x;d)\leq 0\;\forall i\in I(\bar x)\}
\]
the linearization cone to $X$ at $\bar x$.
Above, we used $I(\bar x):=\{i\in\{1,\ldots,m\}\,|\,\theta_i(\bar x)=0\}$.
Observe that in case where the functionals $\theta_i$, $i\in I(\bar x)$, are continuously differentiable
at $\bar x$, then $\mathcal L_X(\bar x)$ coincides with the classical linearization cone of standard
nonlinear programming and is polyhedral. In the more general case discussed above, $\mathcal T_X(\bar x)$
is only a closed cone which does not need to be convex since the lower Dini directional derivative is
only positively homogeneous but not necessarily linear w.r.t.\ the direction.
By standard arguments, we obtain that $\mathcal L_X(\bar x)$ provides an upper estimate of $\mathcal T_X(\bar x)$.
\begin{lemma}\label{lem:Dini_linearization_cone}
	Let $X$ be given as in \eqref{eq:feasible_set_via_inequalities}.
	Fix $\bar x\in X$ and assume that $\theta_1,\ldots,\theta_m$ are locally Lipschitz
	continuous at $\bar x$.
	Then, we have $\mathcal T_X(\bar x)\subseteq\mathcal L_X(\bar x)$.
\end{lemma}

Clearly, the converse inclusion $\mathcal T_X(\bar x)\supseteq\mathcal L_X(\bar x)$
holds in general only under additional assumptions. However, equality would be beneficial
in order to obtain a first-order sufficient optimality condition in \cref{prop:abstract_first_order_sufficient_condition}
which is as weak as possible. Thinking about standard nonlinear programming,
one might be tempted to say that Abadie's Constraint Qualification (ACQ) holds at $\bar x$
whenever it holds $\mathcal T_X(\bar x)=\mathcal L_X(\bar x)$.
In case where the functions $\theta_1,\ldots,\theta_m$ are continuously differentiable,
there exist several constraint qualifications stronger than ACQ which are stated in
terms of initial problem data. In this manuscript, we refer to the Linear Independence
Constraint Qualification (LICQ), the Mangasarian--Fromovitz constraint qualification
(MFCQ), and the Constant Rank Constraint Qualification (CRCQ) in some situations,
see, e.g., \cite{BazaraaSheraliShetty1993,Janin1984} for the definitions and some
discussion.

Using \cref{lem:Dini_linearization_cone}, we obtain the following corollary from
\cref{prop:abstract_first_order_sufficient_condition}.
\begin{corollary}\label{cor:abstract_first_order_sufficient_condition_inequality_constraints}
	Let $\bar x\in\dom\theta$ be a feasible point of \eqref{eq:nonlinear_problem} where $X$
	is given as in \eqref{eq:feasible_set_via_inequalities}. Furthermore, assume that the
	functionals $\theta_1,\ldots,\theta_m$ are locally Lipschitz continuous at $\bar x$.
	Finally, suppose that the system
	\begin{align*}
		\theta^\ast(\bar x;d)&\leq 0\\
		\theta^-_i(\bar x;d)&\leq 0\quad i\in I(\bar x)
	\end{align*}
	does not possess a nontrivial solution.
	Then, the first-order growth condition holds at $\bar x$ for \eqref{eq:nonlinear_problem}.
\end{corollary}

Below, we combine the concepts of second-order directional differentiability and second-order epi-regularity in
order to state a sufficient second-order optimality condition which addresses \eqref{eq:nonlinear_problem}
whenever $X:=\R^n$ holds. We note, however, that this statement can be extended to inequality-constrained
mathematical problems and even to bilevel optimization as we will see in \cref{sec:second_order_optimality_conditions}.
The result and its proof can be found in
\cite[Proposition~2.1]{RueckmannShapiro2001}.
\begin{proposition}\label{prop:second_order_sufficient_condition_unconstrained}
	Set $X:=\R^n$.
	Fix $\bar x\in\dom\theta$ where
	$\theta$ is second-order directionally differentiable and second-order epi-regular.
	Furthermore, assume that
	\[
		\forall d\in\R^n\colon\quad
		\theta'(\bar x;d)\geq 0
	\]
	as well as
	\[
			\forall d\in\{r\in\R^n\,|\,\theta'(\bar x;r)=0\}\setminus\{0\}\colon\quad
			\inf\limits_{w\in\R^n}\theta''(\bar x;d,w)>0
	\]
	hold.
	Then, the second-order growth condition holds at $\bar x$ for \eqref{eq:nonlinear_problem}.
\end{proposition}

\section{Preliminaries from bilevel optimization}\label{sec:preliminaries_bilevel}

Let us provide some notation addressing \eqref{eq:BPP}. 
For later use, we introduce the lower level feasible set mapping $K\colon\R^n\rightrightarrows\R^m$ by
\[
	\forall x\in\R^n\colon\quad
		K(x):=\{y\in\R^m\,|\,g(x,y)\leq 0\}.
\]
Let $M\subseteq\R^n\times\R^m$ denote the feasible set of problem \eqref{eq:BPP} and 
$L\colon\R^n\times\R^m\times\R^q\to\R$ the lower level Lagrangian function defined by
\[
	\forall (x,y,\lambda)\in\R^n\times\R^m\times\R^q\colon\quad
	L(x,y,\lambda):=f(x,y)+\lambda^\top g(x,y).
\]
Using the latter notation, we can define the set of lower level Lagrange multipliers as
\[
	\Lambda(x,y):=\{\lambda\in\R^q\,|\,\nabla_yL(x,y,\lambda)=0,\,\lambda\geq 0,\,\lambda^\top g(x,y)=0\}
\]
for points $(x,y)\in\gph K$. For a point $(\bar x,\bar y)\in\gph K$, we use
\[
	\bar I^g=I^g(\bar x,\bar y):=\{i\in\{1,\ldots,q\}\,|\,g_i(\bar x,\bar y)=0\}
\]
in order to represent the lower level index set of active constraints w.r.t.\ $(\bar x,\bar y)$.
If $(\bar x,\bar y)$ satisfies $G(\bar x,\bar y)\leq 0$, the following definition is reasonable as well:
\[
		\bar I^G=I^G(\bar x,\bar y):=\{i\in\{1,\ldots,p\}\,|\,G_i(\bar x,\bar y)=0\}.
\]

The following lower level value function reformulation of problem \eqref{eq:BPP}
will be important for the developments in this paper:
\begin{equation}\label{eq:LLVF}\tag{LLVF}
    \min\limits_{x,y}\{F(x,y)\,|\,G(x,y)\leq 0, \; f(x,y)\leq \varphi(x), \; g(x,y)\leq 0 \}.
\end{equation}
This is a nonsmooth and nonconvex constrained optimization problem, given that the involved lower level value function $\varphi$ defined by
\begin{equation}\label{eq:varphi}
   \forall x\in\R^n\colon\quad \varphi(x) := \inf\limits_y\{f(x,y)\,|\,g(x,y)\leq 0\}
\end{equation}
is typically nondifferentiable. Even when all the functions involved in \eqref{eq:BPP} are fully convex, i.e., 
$f$ and $g$ are convex w.r.t.\
$(x,y)$, the feasible set of problem \eqref{eq:LLVF} is still typically nonconvex.
Furthermore, it is folklore that \eqref{eq:LLVF} is irregular in the sense that standard constraint qualifications
from nonsmooth programming do not hold at \emph{all} feasible points of \eqref{eq:LLVF}.
However, we note that \eqref{eq:BPP} and \eqref{eq:LLVF} are fully equivalent optimization problems.

We fix some point $(\bar x,\bar y)\in\gph K$.
Generally, we say that a constraint qualification holds for the lower level problem \eqref{eq:lower-level}
at $(\bar x,\bar y)$ if it is valid at the point $\bar y$ within the set $K(\bar x)$ for fixed parameter.
Exemplary, the lower level linear independence constraint qualification (LLICQ) will be said to hold at $(\bar x, \bar y)$
if the family $\{\nabla_yg_i(\bar x,\bar y)\}_{i\in\bar I^g}$ of gradients is linearly independent.
Supposing that $(\bar x,\bar y)\in\gph S$ is given such that a constraint qualification holds for \eqref{eq:lower-level}
at this point, the set $\Lambda(\bar x,\bar y)$ is nonempty.
For $\lambda\in\Lambda(\bar x,\bar y)$, we introduce the lower level critical cone at $(\bar x,\bar y)$ as stated below:
\[
	\mathcal C^L(\bar x,\bar y)
	:=
	\left\{
		\delta_y\in\R^m
		\,\middle|\,
		\begin{aligned}
			\nabla_yg_i(\bar x,\bar y)^\top\delta_y&\,\leq\, 0\;&&i\in\bar I^g,\,\lambda_i=0\\
			\nabla_yg_i(\bar x,\bar y)^\top\delta_y&\,=\,0\;&&i\in\bar I^g,\,\lambda_i>0
		\end{aligned}
	\right\}.
\]
It is well known that this definition actually does not depend on the particular choice of the multiplier since we also have
\begin{equation}\label{eq:lower_level_critical_cone}
	\mathcal C^L(\bar x,\bar y)
	=
	\left\{
		\delta_y\in\R^m
		\,\middle|\,
		\begin{aligned}
			\nabla_yf(\bar x,\bar y)^\top\delta_y&\,=\, 0\;&&\\
			\nabla_yg_i(\bar x,\bar y)^\top\delta_y&\,\leq\,0\;&&i\in\bar I^g
		\end{aligned}
	\right\}
\end{equation}
by elementary calculations exploiting the definition of $\Lambda(\bar x,\bar y)$.
Note that under validity of LLICQ, the set $\Lambda(\bar x,\bar y)$ is a singleton.
In this situation, lower level Strict Complementarity (LSC) is said to hold whenever $\bar\lambda_j>0$
is valid for all $j\in\bar I^g$ where $\bar\lambda$ is the unique element from $\Lambda(\bar x,\bar y)$.
Finally, the lower level Second-Order Sufficient Condition (LSOSC) is said to hold at $(\bar x, \bar y)$ if we have
\[
	\forall \delta_y\in\mathcal C^L(\bar x,\bar y)\setminus\{0\}
	\;\, \exists\lambda\in\Lambda(\bar x,\bar y)\colon\quad
	\delta_y^{\top}\nabla_{yy}^2 L(\bar x, \bar y, \lambda)\delta_y>0.
\]

Next, we briefly recall a partial exact penalization principle addressing \eqref{eq:LLVF} popular in bilevel programming.
Therefore, we fix a local minimizer $(\bar x,\bar y)\in\R^n\times\R^m$ of \eqref{eq:BPP}.
For a neighborhood $U\subseteq\R^n\times\R^m\times\R$ of $(\bar{x} ,\bar{y},0)$
and define the set of locally feasible points of the problem resulting from a perturbation on the value function of problem \eqref{eq:LLVF} by
\[
	\mathcal{F}_U
	:=
	\left\{
		(x,y,\varsigma)\in U\,|\,
		G(x,y)\leq 0,\,f(x,y)- \varphi(x)+\varsigma=0,\,g(x,y)\leq 0
	\right\}.
\]
Problem \eqref{eq:LLVF} will be said to be \emph{partially calm} at $(\bar{x} ,\bar{y})$
if there exist $\kappa>0$ and a neighborhood $U\subseteq\R^n\times\R^m\times\R$ of $(\bar{x},\bar{y},0)$ such that
\[
	\forall (x,y,\varsigma)\in\mathcal F_U\colon\quad
  	F(x,y)-F(\bar{x},\bar{y})+\kappa |\varsigma|\geq 0.
\]
Partial calmness dates back to the seminal paper \cite{YeZhu1995} where the authors show that \eqref{eq:LLVF} is
partially calm at one of its local minimizers $(\bar x,\bar y)$ if and only if there is some $\bar\kappa>0$ such that
$(\bar x,\bar y)$ is also a local minimizer of the problem
\begin{equation}\label{eq:partially_penalized_LLVF}
	\min\limits_{x,y}\{F(x,y)+\kappa(f(x,y)-\varphi(x))\,|\,G(x,y)\leq 0,\,g(x,y)\leq 0\}
\end{equation}
for each $\kappa\geq\bar\kappa$.
Clearly, dealing with \eqref{eq:partially_penalized_LLVF} is beneficial since standard constraint qualifications may
hold at the feasible points of this problem. As a consequence, KKT-type necessary optimality conditions, which exploit
some generalized derivative of $\varphi$, can be used to characterize the local minimizer $(\bar x,\bar y)$.
Observe that the resulting KKT-system may be interpreted as the KKT-system of \eqref{eq:LLVF} where the
penalization parameter $\kappa$ plays the role of the Lagrange multiplier associated with the constraint
$f(x,y)-\varphi(x)\leq 0$.
In \cite[Section~4]{DempeZemkoho2013}, \cite[Section~6]{Mordukhovich2018}, and \cite{YeZhu1995}, the authors present
several conditions which ensure validity of partial calmness at local minimizers.
Exemplary, let us mention that bilevel optimization problems with fully linear lower level are partially calm at
all their local minimizers due to \cite[Proposition~4.1]{YeZhu1995}.
In general, however, as highlighted in \cite{HenrionSurowiec2011}, this property is very restrictive.

\section{First-order sufficient optimality conditions}\label{sec:first_order_sufficient_conditions}

As already mentioned in \cref{sec:prelimiaries_nonlinear_programming}, the term \emph{first-order}
refers to the fact that the optimality conditions derived in this section imply the fulfillment of 
the first-order growth condition at the reference point for \eqref{eq:BPP}.
As we will see, depending on the exploited approach, the obtained sufficient optimality
conditions may contain first- or second-order derivative information of the lower level program, i.e.\
we would like to make clear that a first-order sufficient optimality condition for \eqref{eq:BPP}
still may comprise derivative information of order higher than one.

From \cref{prop:abstract_first_order_sufficient_condition}, we obtain the following general sufficient optimality
condition which will be the base of our analysis in this section.
Recall that the set $M$ denotes the feasible set of \eqref{eq:BPP}.
\begin{lemma}\label{lem:trivial_first_order_condition}
	Let $(\bar x,\bar y)\in\R^n\times\R^m$ be a feasible point of \eqref{eq:BPP} such that the system
	\[
		\nabla F(\bar x,\bar y)^\top d\leq 0
		\qquad
		d\in\mathcal T_M(\bar x,\bar y)
	\]
	does not possess a nontrivial solution. Then, the first-order growth condition holds at
	$(\bar x,\bar y)$ for \eqref{eq:BPP}. Particularly, $(\bar x,\bar y)$ is a strict local minimizer of \eqref{eq:BPP}.
\end{lemma}

The above lemma is of limited practical use as long as no reasonable upper estimate of $\mathcal T_M(\bar x,\bar y)$
in terms of initial problem data is available. Here, we are going to discuss such estimates.

\subsection{Value-function approach}

Using the optimal value function $\varphi$, $M$ possesses the equivalent representation
\[
	M =\left\{(x,y)\in\R^n\times\R^m\,|\,G(x,y)\leq 0, \; f(x,y)\leq \varphi(x), \; g(x,y)\leq 0 \right\}.
\]
Thus, the computation of an upper estimate of $\mathcal T_M(\bar x,\bar y)$ can be carried with the
aid of \cref{lem:Dini_linearization_cone}, and \cref{cor:abstract_first_order_sufficient_condition_inequality_constraints}
yields a sufficient optimality condition for \eqref{eq:BPP}.
\begin{theorem}\label{thm:first_order_sufficient_condition}
	Let $(\bar x,\bar y)\in\R^n\times\R^m$ be a feasible point of \eqref{eq:BPP}.
	Assume that $K$ is locally bounded at $\bar x$ while LMFCQ holds at all
	points $(\bar x,y)\in\gph S$.
	Furthermore, assume that the system
	\begin{subequations}\label{eq:first_order_sufficient_condition_system_abstract}
		\begin{align}
			\label{eq:first_order_sufficient_condition_system_abstract_F}
			\nabla F(\bar x,\bar y)^\top d&\,\leq\,0\\
			\label{eq:first_order_sufficient_condition_system_abstract_G}
			\nabla G_i(\bar x,\bar y)^\top d&\,\leq\,0
				\quad i\in \bar I^G\\
			\label{eq:first_order_sufficient_condition_system_abstract_f}
			\nabla f(\bar x,\bar y)^\top d-\varphi^+(\bar x;\delta_x)&\,\leq\,0\\
			\label{eq:first_order_sufficient_condition_system_abstract_g}
			\nabla g_i(\bar x,\bar y)^\top d&\,\leq\,0\quad i\in \bar I^g
		\end{align}
	\end{subequations}
	does not possess a nontrivial solution $d :=(\delta_x,\delta_y)\in\R^n\times\R^m$.
	Then, $(\bar x,\bar y)$ is a strict local minimizer of \eqref{eq:BPP}.
\end{theorem}
\begin{proof}
	The assumptions of the theorem guarantee that $\varphi$ is locally Lipschitz continuous
	at $\bar x$, see, e.g., \cite[Theorem~4.17]{Dempe2002}.
	Thus, we can apply \cref{cor:abstract_first_order_sufficient_condition_inequality_constraints} to \eqref{eq:LLVF}.
	The result follows observing that
	\begin{align*}
		(f-\varphi)^-((\bar x,\bar y);d)
		=
		\nabla f(\bar x,\bar y)^\top d+(-\varphi)^-(\bar x;\delta_x)
		=
		\nabla f(\bar x,\bar y)^\top d-\varphi^+(\bar x;\delta_x)
	\end{align*}
	holds by definition of the upper Dini directional derivative
	for all $d :=(\delta_x,\delta_y)\in\R^n\times\R^m$.
\end{proof}
\begin{example}\label{ex:first_order_sufficient_condition}
	 Let us consider the bilevel programming problem
	 \[
	 	\min\limits_{x,y}\left\{\tfrac12(x+3)^2+\tfrac12(y+1)^2\,\middle|\,y\in S(x)\right\}
	 \]
	 where $S\colon\R\rightrightarrows\R$ is the solution set mapping of
	 $\min_y\{xy\,|\,0\leq y\leq 1\}$.
	 Noting that the lower level feasible set is independent of $x$, compact, and regular,
	 the associated optimal value function is locally Lipschitz continuous.
	 Indeed, a simple calculation shows
	 \[
	 	\forall x\in\R\colon\quad
	 	S(x)=
	 		\begin{cases}
	 			\{1\}	&x<0,\\
	 			[0,1]	&x=0,\\
	 			\{0\}	&x>0,
	 		\end{cases}
	 	\qquad \mbox{ and }\qquad
		\varphi(x)=\min\{x;0\}.
	 \]
	Let us consider the point $(\bar x,\bar y):=(0,0)$.
	The associated system \eqref{eq:first_order_sufficient_condition_system_abstract} reads as
	 \begin{align*}
	 	3\delta_x+\delta_y \leq 0,	\quad
 		-\min\{\delta_x;0\}\leq 0,	\quad
	 	-\delta_y&\,\leq\,0.
	 \end{align*}
	 Clearly, this system possesses only the trivial solution $(\delta_x,\delta_y)=(0,0)$ which is why
	 $(\bar x,\bar y)$ is a strict local minimizer of the considered bilevel optimization problem.
\end{example}

Below, we present some corollaries of \cref{thm:first_order_sufficient_condition} where the
abstract upper Dini derivative of $\varphi$ is approximated in terms of initial problem data.
Under the assumptions of \cref{thm:first_order_sufficient_condition}, the upper Dini directional
derivative of $\varphi$ can be estimated from above, see \cite[Theorem~4.15]{Dempe2002}.
Using this approximate, the following corollary follows easily.
\begin{corollary}\label{cor:approximate_upper_Dini_derivative_from_below}
	Let $(\bar x,\bar y)\in\R^n\times\R^m$ be a feasible point of \eqref{eq:BPP}.
	Assume that $K$ is locally bounded at $\bar x$ while LMFCQ holds at all
	points $(\bar x,y)\in\gph S$. Finally, assume that the system
	\eqref{eq:first_order_sufficient_condition_system_abstract_F},
	\eqref{eq:first_order_sufficient_condition_system_abstract_G},
	\eqref{eq:first_order_sufficient_condition_system_abstract_g}, and
	\[
		\nabla_yf(\bar x,\bar y)^\top \delta_y
		-\inf\limits_{y\in S(\bar x)}\max\limits_{\lambda\in\Lambda(\bar x,y)}
		\left((\nabla_xf(\bar x,y)-\nabla_xf(\bar x,\bar y))^\top \delta_x+\lambda^\top\nabla_xg(\bar x,y)\delta_x\right)
		\leq 0
	\]
	does not possess a nontrivial solution $d :=(\delta_x,\delta_y)\in\R^n\times\R^m$.
	Then, $(\bar x,\bar y)$ is a strict local minimizer of \eqref{eq:BPP}.
\end{corollary}

In case where the lower level problem is convex w.r.t.\ $y$, i.e., if the lower level data functions
$f$ and $g_1,\ldots,g_q$ are convex w.r.t.\ $y$ for each choice of the parameter $x$, the assumptions of
\cref{thm:first_order_sufficient_condition} already guarantee that $\varphi$ is
directionally differentiable at the reference point, see \cite[Theorem~4.16]{Dempe2002}.
This result already provides a ready-to-use formula for the directional derivative
which yields a slightly better result than the one presented in
\cref{cor:approximate_upper_Dini_derivative_from_below}.
\begin{corollary}
	Let $(\bar x,\bar y)\in\R^n\times\R^m$ be a feasible point of \eqref{eq:BPP}
	where the lower level problem is convex w.r.t.\ $y$.
	Assume that $K$ is locally bounded at $\bar x$ while LMFCQ holds at all
	points $(\bar x,y)\in\gph S$. Finally, assume that the system
	\eqref{eq:first_order_sufficient_condition_system_abstract_F},
	\eqref{eq:first_order_sufficient_condition_system_abstract_G},
	\eqref{eq:first_order_sufficient_condition_system_abstract_g}, and
	\[
		\nabla_yf(\bar x,\bar y)^\top \delta_y
		-\min\limits_{y\in S(\bar x)}\max\limits_{\lambda\in\Lambda(\bar x,y)}
		\left((\nabla_xf(\bar x,y)-\nabla_xf(\bar x,\bar y))^\top \delta_x+\lambda^\top\nabla_xg(\bar x,y)\delta_x\right)
		\leq 0
	\]
	does not possess a nontrivial solution $d :=(\delta_x,\delta_y)\in\R^n\times\R^m$.
	Then, $(\bar x,\bar y)$ is a strict local minimizer of \eqref{eq:BPP}.
\end{corollary}

Finally, the situation of a fully convex lower level problem is even more
comfortable since we do not need to consider the minimum over all
lower level solutions associated with the reference point, see
\cite[Corollary~4.7]{Dempe2002}.
\begin{corollary}
	Let $(\bar x,\bar y)\in\R^n\times\R^m$ be a feasible point of \eqref{eq:BPP}
	where the lower level problem is jointly convex w.r.t.\ $(x,y)$.
	Assume that $K$ is locally bounded at $\bar x$ while LMFCQ holds at all
	points $(\bar x,y)\in\gph S$. Finally, assume that the system
	\eqref{eq:first_order_sufficient_condition_system_abstract_F},
	\eqref{eq:first_order_sufficient_condition_system_abstract_G},
	\eqref{eq:first_order_sufficient_condition_system_abstract_g}, and
	\[
		\nabla_yf(\bar x,\bar y)^\top \delta_y
		-\max\limits_{\lambda\in\Lambda(\bar x,\bar y)}
		\lambda^\top\nabla_xg(\bar x,\bar y)\delta_x
		\leq 0
	\]
	does not possess a nontrivial solution $d :=(\delta_x,\delta_y)\in\R^n\times\R^m$.
	Then, $(\bar x,\bar y)$ is a strict local minimizer of \eqref{eq:BPP}.
\end{corollary}

As soon as the negative value function $-\varphi$ is regular in Clarke's sense,
the assertion of \cref{thm:first_order_sufficient_condition}
can be stated in terms of Clarke's subdifferential of $\varphi$.
\begin{theorem}\label{thm:first_order_sufficient_condition_dual}
	Let $(\bar x,\bar y)\in\R^n\times\R^m$ be a feasible point of \eqref{eq:BPP}.
	Assume that $\varphi$ is locally Lipschitz continuous at $\bar x$
	while $-\varphi$ is Clarke-regular at $\bar x$.
	Furthermore, suppose that the polar cone of the set
	\begin{equation}\label{eq:definition_Q}
		\begin{aligned}
		\mathcal Q:=\{\nabla F(\bar x,\bar y)\}
		&\cup
		\{\nabla G_i(\bar x,\bar y)\,|\,i\in \bar I^G\}\\
		&\cup
		\left\{\nabla f(\bar x,\bar y)-\begin{bmatrix}\xi\\0\end{bmatrix}\,\middle|\,\xi\in\partial^c\varphi(\bar x)\right\}
		\cup
		\{\nabla g_i(\bar x,\bar y)\,|\,i\in \bar I^g\}
	\end{aligned}
	\end{equation}
	reduces to zero. Then, $(\bar x,\bar y)$ is a strict local minimizer of \eqref{eq:BPP}.
\end{theorem}
\begin{proof}
	The properties of Clarke's subdifferential can be used in order to see that
	\begin{align*}
		\mathcal Q=\{\nabla F(\bar x,\bar y)\}
		&\cup
		\{\nabla G_i(\bar x,\bar y)\,|\,i\in \bar I^G\}\\
		&\cup
		\left\{\nabla f(\bar x,\bar y)+\begin{bmatrix}\tilde \xi\\0\end{bmatrix}\,\middle|\,\tilde \xi\in\partial^c(-\varphi)(\bar x)\right\}
		\cup
		\{\nabla g_i(\bar x,\bar y)\,|\,i\in \bar I^g\}
	\end{align*}
	holds true.
	Since $\mathcal Q^\circ=\{0\}$ is valid by assumption, the system
	\eqref{eq:first_order_sufficient_condition_system_abstract_F}, \eqref{eq:first_order_sufficient_condition_system_abstract_G},
	\eqref{eq:first_order_sufficient_condition_system_abstract_g}, and
	\begin{align*}
		\nabla f(\bar x,\bar y)^\top d
		+\max\{\tilde \xi^\top \delta_x\,|\,\tilde \xi\in\partial^c(-\varphi)(\bar x)\}
		\leq 0
	\end{align*}
	does not possess a nontrivial solution $d:=(\delta_x,\delta_y)\in\R^n\times\R^m$.
	Observing that $-\varphi$ is Clarke-regular and, therefore, directionally differentiable,
	at $\bar x$, we particularly have
	\begin{align*}
		\max\{\tilde \xi^\top \delta_x\,|\,\tilde \xi\in\partial^c(-\varphi)(\bar x)\}
		=
		(-\varphi)^\circ(\bar x;\delta_x)
		=
		(-\varphi)'(\bar x;\delta_x)
		=
		-\varphi'(\bar x;\delta_x)
		=		
		-\varphi^+(\bar x;\delta_x).
	\end{align*}
	Thus,  \eqref{eq:first_order_sufficient_condition_system_abstract} does not possess a nontrivial solution.
	Due to \cref{thm:first_order_sufficient_condition}, the assertion follows.
\end{proof}

Clearly, the assumptions on the function $\varphi$ which are postulated in \Cref{thm:first_order_sufficient_condition_dual}
trivially hold whenever $\varphi$ is locally finite and concave around the reference point.
Exemplary, assume for a moment that the lower level problem is given by the simple parametric
linear program
\[
	\min\limits_y\{(Ax+c)^\top y\,|\,By\leq b\}
\]
for matrices $A\in\R^{m\times n}$, $B\in\R^{q\times m}$, $b\in\R^q$, and $c\in\R^m$.
Then, the associated optimal value function is concave and one obtains
\[
	\partial^c\varphi(\bar x)=A^\top S(\bar x)=\{A^\top y\,|\,By\leq b,\,B^\top p+A\bar x+c=0,\,p\geq 0,\,p^\top(By-b)=0\}
\]
by strong duality of linear programming whenever $\bar x\in\intr\dom S$ holds true,
see \cref{sec:linear_lower_level_param_in_obj} as well.
This means that $\mathcal Q$ in \eqref{eq:definition_Q} is easily given in terms
of initial problem data in this case.

We would like to point the reader's attention to the observation that the polar cone
of the set $\mathcal Q$ given in \eqref{eq:definition_Q} reduces to $\{0\}$ whenever
$0\in\intr\conv\mathcal Q$ holds which might be easier to check than calculating
the overall cone $\mathcal Q^\circ$ or solving the system
\eqref{eq:first_order_sufficient_condition_system_abstract}.

\begin{example}\label{ex:first_order_sufficient_condition_continued}
	We revisit \Cref{ex:first_order_sufficient_condition} w.r.t.\ the point $(\bar x,\bar y):=(0,0)$
	in the light of \Cref{thm:first_order_sufficient_condition_dual}.
	This is possible since $\varphi$ is a concave function.
	The set $\mathcal Q$ is given by
	\[
		\mathcal Q=
		\left\{
		\begin{pmatrix}
			3\\2
		\end{pmatrix}
		\right\}
		\cup
		\left\{
		\begin{pmatrix}
			\tilde \xi\\0
		\end{pmatrix}
		\,\middle|\,
		-1\leq\tilde \xi\leq 0
		\right\}
		\cup
		\left\{
		\begin{pmatrix}
		0\\-1
		\end{pmatrix}
		\right\}
	\]
	in this context. Given that $0\in\intr\conv\mathcal Q$ holds, $\mathcal Q^\circ$ reduces to the origin.
\end{example}

\subsection{Stable unique lower level solutions}\label{sec:first_order_stable_lower_level_solutions}

Let $M_u:=\{(x,y)\in\R^n\times\R^m\,|\,G(x,y)\leq 0\}$ denote the upper level feasible set.
Clearly, we have $M=\gph S\cap M_u$ which yields
\begin{equation}\label{eq:decomposed_bilevel_feasible_set}
	\mathcal T_M(x,y)\subseteq\mathcal T_{\gph S}(x,y)\cap\mathcal T_{M_u}(x,y)
\end{equation}
for each $(x,y)\in M$, see \cite[Table~4.1]{AubinFrankowska2009}.
Clearly, $\mathcal T_{M_u}(x,y)$ can be expressed via the associated linearized tangent cone
as long as at least ACQ is valid at $(x,y)$ w.r.t.\ $M_u$. Thus, it remains to approximate
$\mathcal T_{\gph S}(x,y)$ from above.

In this section, we consider the situation where $S$ is locally single-valued, Lipschitz continuous,
and directionally differentiable around a reference point $\bar x$, i.e.\ there is a
neighborhood $U\subseteq\R^n$ of $\bar x$ and a locally Lipschitz continuous as well as
directionally differentiable function $s\colon U\to\R^m$ satisfying $S(x)=\{s(x)\}$ for
all $x\in U$. This can be guaranteed if the lower level problem is convex w.r.t.\ $y$ and
satisfies LMFCQ, LCRCQ, as well as
a strong second-order sufficient optimality condition at $(\bar x,\bar y)\in\gph S$,
and under these assumptions, the directional derivative of $s$ can be computed as the solution
of a quadratic optimization problem which is given in terms of initial problem data, 
see \cite[Theorem~2]{RalphDempe1995} for details.
Anyway, exploiting the directional differentiability of $s$ at $\bar x$, we have
\[
	\mathcal T_{\gph S}(\bar x,\bar y)
	=
	\mathcal T_{\gph s}(\bar x,s(\bar x))
	=
	\{(\delta_x,s'(\bar x;\delta_x))\in\R^n\times\R^m\,|\,\delta_x\in\R^n\}
\]
by standard arguments. This leads to the following sufficient optimality condition
for \eqref{eq:BPP} which recovers \cite[Theorem~4.2]{Dempe1992} and the considerations
in \cite[Section~2]{DempeKalashnikovKalashnykova2006}.
\begin{theorem}\label{thm:locally_single_valued_solution_map}
	Fix a feasible point $(\bar x,\bar y)\in\R^n\times\R^m$ of \eqref{eq:BPP}.
	Assume that there is a neighborhood $U\subseteq\R^n$ of $\bar x$ and a locally Lipschitz
	continuous, directionally differentiable function $s\colon U\to\R^m$ such that
	$S(x)=\{s(x)\}$ is valid for all $x\in U$.
	Furthermore, assume that the system
	\begin{equation}\label{eq:OC_single_valued_solution_map}
		\begin{aligned}
			\nabla_xF(\bar x,\bar y)^\top \delta_x+\nabla _yF(\bar x,\bar y)^\top s'(\bar x;\delta_x)&\,\leq\, 0,\\
			\nabla_xG_i(\bar x,\bar y)^\top \delta_x+\nabla _yG_i(\bar x,\bar y)^\top s'(\bar x;\delta_x)&\,\leq\,0
			\quad\forall i\in \bar I^G
		\end{aligned}
	\end{equation}
	does not possess a nontrivial solution $\delta_x\in\R^n$.
	Then, $(\bar x,\bar y)$ is a strict local minimizer of \eqref{eq:BPP}.
\end{theorem}
\begin{proof}
	Noting that $\mathcal T_M(\bar x,\bar y)$ can be approximated from above by the set
	$\mathcal T_{\gph S}(\bar x,\bar y)\cap\mathcal T_{M_u}(\bar x,\bar y)$ which admits the
	representation	
	\[
		\left\{(\delta_x,s'(\bar x;\delta_x))\in\R^n\times\R^m
		\,\middle|\,
			\nabla_xG_i(\bar x,\bar y)^\top \delta_x+\nabla _yG_i(\bar x,\bar y)^\top s'(\bar x;\delta_x)\,\leq\,0
			\;\forall i\in \bar I^G
		\right\}
	\]
	due to the above considerations, this is an immediate consequence of \cref{lem:trivial_first_order_condition}.
\end{proof}

\subsection{Variational analysis approach}

In this section, we want to exploit the decomposition \eqref{eq:decomposed_bilevel_feasible_set} without
assuming that $S$ is locally single-valued around the reference point.
Therefore, let us impose the following assumption for the remaining
part of this section.
\begin{assumption}\label{ass:variational_analysis_approach}
	For each $x\in\R^n$, $f(x,\cdot)\colon\R^m\to\R$ is convex. Furthermore, the function $g$
	is independent of $x$ and componentwise convex.
	Thus, it is reasonable to set $K:=\{y\in\R^m\,|\,g(y)\leq 0\}$.
\end{assumption}
The above assumption allows us to interpret $S$ as the solution map associated with an equilibrium
problem in the following sense:
\[
	\forall x\in\R^n\colon\quad
	S(x)=\{y\in\R^m\,|\,-\nabla_yf(x,y)\in\widehat{\mathcal N}_{K}(y)\}.
\]
This leads to
\[
	\gph S=
	\left\{
		(x,y)\in\R^n\times\R^m\,\middle|\,
		(y,-\nabla_yf(x,y))\in\gph\widehat{\mathcal N}_K
	\right\}
\]
where $\widehat{\mathcal N}_K\colon\R^m\rightrightarrows\R^m$ represents the (regular) normal
cone mapping associated with the set $K$, which assigns to each $y\in\R^m$
the set $\widehat{\mathcal N}_{K}(y)$. By convexity and closedness of $K$, $\gph\widehat{\mathcal N}_K$
is closed as well. Exploiting the preimage rule
from \cite[Theorem~6.31]{RockafellarWets1998} as well as the fact that metric subregularity of
feasibility maps is sufficient for some generalized ACQ to hold,
see \cite[Proposition~5]{HenrionOutrata2005}, we obtain the following result.
\begin{lemma}\label{lem:upper_approximate_tangent_cone_graph_S}
	For $(\bar x,\bar y)\in\gph S$, we have
	\[
		\mathcal T_{\gph S}(\bar x,\bar y)
		\subseteq
		\left\{
			(\delta_x,\delta_y)\in\R^n\times\R^m\,\middle|\,
			(\delta_y,-\nabla^2_{yx}f(\bar x,\bar y)\delta_x-\nabla^2_{yy}f(\bar x,\bar y)\delta_y)
			\in
			\mathcal T_{\gph\widehat{\mathcal N}_K}(\bar y,-\nabla_yf(\bar x,\bar y))
		\right\}.
	\]
	Equality holds whenever the feasibility mapping
	\begin{equation}\label{eq:feasibility_map_VI_approach}
		\R^n\times\R^m\ni(x,y)\mapsto\gph\widehat{\mathcal N}_K-\{(y,-\nabla_y f(x,y))\}\subseteq\R^m\times\R^m
	\end{equation}
	is metrically subregular at $(\bar x,\bar y,0,0)$.
\end{lemma}

Recently, some progress has been made regarding the characterization of the tangent cone
to the graph of $\widehat{\mathcal N}_K$, see \cite[Section~4]{GfrererOutrata2016}.
Thus, combining \eqref{eq:decomposed_bilevel_feasible_set} with
\cref{lem:upper_approximate_tangent_cone_graph_S}, \cref{thm:first_order_sufficient_condition}
allows the derivation of a first-order sufficient optimality condition for \eqref{eq:BPP}.
\begin{theorem}\label{thm:first_order_sufficient_condition_var_analysis}
	Let $(\bar x,\bar y)\in\R^n\times\R^m$ be a feasible point of \eqref{eq:BPP}.
	Assume for each $u\in\mathcal L_K(\bar y)\setminus\{0\}$ that
	\begin{equation}\label{eq:SOSCMS}
		0=\nabla g(\bar y)^\top\lambda,\,\lambda\geq 0,\,\lambda^\top g(\bar y)=0,\,
		\sum\limits_{i=1}^q\lambda_iu^\top\nabla^2 g_i(\bar y)u\geq 0
		\quad\Longrightarrow\quad
		\lambda=0.
	\end{equation}
	Finally, suppose that the system \eqref{eq:first_order_sufficient_condition_system_abstract_F},
	\eqref{eq:first_order_sufficient_condition_system_abstract_G}, and
	\begin{subequations}\label{eq:first_order_sufficient_VA}
		\begin{align}
			\label{eq:first_order_sufficient_VA_graphS}
				\nabla^2_{yx}f(\bar x,\bar y)\delta_x+\nabla^2_{yy}f(\bar x,\bar y)\delta_y
				+
				\kappa\nabla_yf(\bar x,\bar y)
				+
				\sum\limits_{i=1}^q\bigl(\lambda_i\nabla^2g_i(\bar y)\delta_y+\mu_i\nabla g_i(\bar y)\bigr)
				&\,=\,0,\\
			\label{eq:first_order_sufficient_VA_lambda}
				\nabla_yf(\bar x,\bar y)+\nabla g(\bar y)^\top\lambda\,=\,0,\quad
				\lambda\geq 0,\quad
				\lambda^\top g(\bar y)&\,=\,0,\\
			\label{eq:first_order_sufficient_VA_mu}
				\mu\geq 0,\quad\mu^\top g(\bar y)\,=\,0,\quad
				\mu^\top\nabla g(\bar y)\delta_y&\,=\,0,\\
			\label{eq:first_order_sufficient_VA_linearization_cone}
				\nabla g_i(\bar y)^\top\delta_y&\,\leq\,0\quad i\in \bar I^g,\\
			\label{eq:first_order_sufficient_VA_critical_cone}
				\nabla f_y(\bar x,\bar y)^\top\delta_y&\,=\,0
		\end{align}
	\end{subequations}
	does not possess a solution $(\delta_x,\delta_y,\lambda,\kappa,\mu)\in\R^n\times\R^m\times\R^q\times\R\times\R^q$ 
	which satisfies $(\delta_x,\delta_y)\neq(0,0)$.
	Then, $(\bar x,\bar y)$ is a strict local minimizer of \eqref{eq:BPP}.
\end{theorem}
\begin{proof}
	Due to \cite[Proposition~3, Theorem~1]{GfrererOutrata2016},
 	the postulated condition \eqref{eq:SOSCMS} guarantees that we have
 	\[
 		\mathcal T_{\gph\widehat{\mathcal N}_K}(\bar y,-\nabla_yf(\bar x,\bar y))
 		=
 		\left\{
 			(r_1,r_2)\in\R^m\times\R^m\,\middle|\,
 			\exists\lambda\in\Lambda(\bar x,\bar y)\colon\,
 			r_2-\sum_{i=1}^q\lambda_i\nabla^2g_i(\bar y)r_1\in\widehat{\mathcal N}_{\mathcal C^L(\bar x,\bar y)}(r_1)
 		\right\}
 	\]
 	where $\Lambda(\bar x,\bar y)$ denotes the set of lower level Lagrange multipliers and
	$\mathcal C^L(\bar x,\bar y)$ represents the lower level critical cone,
	see \eqref{eq:lower_level_critical_cone}.
 	Noting that $C^L(\bar x,\bar y)$ is a closed, convex, polyhedral cone, we have
 	\begin{align*}
 		\widehat{\mathcal N}_{\mathcal C^L(\bar x,\bar y)}(\delta_y)
 		&=
 		\mathcal C^L(\bar x,\bar y)^\circ\cap\{\delta_y\}^\perp\\
 		&=
 		\left\{
 			\kappa\nabla_yf(\bar x,\bar y)+\sum_{i=1}^q\mu_i\nabla g_i(\bar y)\,\middle|\,
 			\kappa\in\R,\,\mu\geq 0,\,\mu^\top g(\bar y)=0
 		\right\}
 		\cap\{\delta_y\}^\perp\\
 		&=
 		\left\{
 			\kappa\nabla_yf(\bar x,\bar y)+\sum_{i=1}^q\mu_i\nabla g_i(\bar y)\,\middle|\,
 			\kappa\in\R,\,\mu\geq 0,\,\mu^\top g(\bar y)=0,\,\mu^\top\nabla g(\bar y)\delta_y=0
 		\right\}
 	\end{align*}
 	for each $\delta_y\in \mathcal C^L(\bar x,\bar y)$ by definition of the latter cone.
 	
 	Clearly, for each $(\delta_x,\delta_y)\in\mathcal T_M(\bar x,\bar y)$,
 	we have \eqref{eq:first_order_sufficient_condition_system_abstract_G}
 	and $(\delta_x,\delta_y)\in\mathcal T_{\gph S}(\bar x,\bar y)$
 	due to \eqref{eq:decomposed_bilevel_feasible_set} and \cref{lem:Dini_linearization_cone}.
 	Combining \cref{lem:upper_approximate_tangent_cone_graph_S} and the above considerations,
 	$(\delta_x,\delta_y)\in\mathcal T_{\gph S}(\bar x,\bar y)$ implies
 	that there are $\lambda,\mu\in\R^q$ and $\kappa\in\R$
 	satisfying \eqref{eq:first_order_sufficient_VA}.
 	Thus, observing that each solution $(\delta_x,\delta_y,\lambda,\kappa,\mu)$ of
 	\eqref{eq:first_order_sufficient_condition_system_abstract_F}, \eqref{eq:first_order_sufficient_condition_system_abstract_G},
 	and \eqref{eq:first_order_sufficient_VA} satisfies $(\delta_x,\delta_y)=(0,0)$,
 	there does not exist $(\delta_x,\delta_y)\in\mathcal T_M(\bar x,\bar y)$ such that
 	\[
 		\nabla _xF(\bar x,\bar y)^\top \delta_x+\nabla _yF(\bar x,\bar y)^\top \delta_y\leq 0.
 	\]
 	Hence, due to \cref{lem:trivial_first_order_condition}, $(\bar x,\bar y)$ is a strict local
 	minimizer of \eqref{eq:BPP}.
\end{proof}

In the literature, \eqref{eq:SOSCMS} is referred to as \emph{Second-Order Sufficient Condition for
Metric Subregularity}, see, e.g., \cite[Corollary~1]{GfrererKlatte2016}. Clearly, this condition is weaker than LMFCQ.
Recall that the exploited estimate for $\mathcal T_{\gph S}(\bar x,\bar y)$ is sharp when
the set-valued mapping in \eqref{eq:feasibility_map_VI_approach} is metrically
subregular, see \cref{lem:upper_approximate_tangent_cone_graph_S}. In this case, the conditions in
\cref{thm:first_order_sufficient_condition_var_analysis} are indeed of reasonable strength.
Note that whenever $K$ is polyhedral, then $\gph\widehat{\mathcal N}_K$ is the union of finitely many
polyhedral sets. If, additionally, $\nabla_yf$ is affine jointly in $x$ and $y$, then the multifunction
in \eqref{eq:feasibility_map_VI_approach} is a so-called polyhedral set-valued mapping, and this
property ensures that it is metrically subregular everywhere on its graph, see
\cite[Proposition~1]{Robinson1981}. This setting is inherent for lower level problems
of obstacle-type given by
\begin{equation}\label{eq:obstacle_problem}
	\min\limits_y\{\tfrac12y^\top Ay-(Bx+c)^\top y\,|\,\psi_\ell\leq y\leq\psi_u\}
\end{equation}
where $A\in\R^{m\times m}$ is symmetric and positive semidefinite, $B\in\R^{m\times n}$ as well as $c\in\R^m$
are chosen arbitrarily, and the lower and upper obstacle $\psi_\ell,\psi_u\in\R^m$ satisfy
$\psi_\ell<\psi_u$, see \cite[Section~4]{HenrionSurowiec2011} as well.

Exploiting recent theory from \cite{GfrererMordukhovich2019}, the results of
\cref{thm:first_order_sufficient_condition_var_analysis} can be generalized
to the setting where $K$ depends on $x$. In this case, some more restrictive assumptions have to be postulated
and the system \eqref{eq:first_order_sufficient_VA} gets far more complex.

The upcoming example depicts that, in contrast to \cref{thm:locally_single_valued_solution_map},
\cref{thm:first_order_sufficient_condition_var_analysis} is indeed capable
of identifying local minimizers of \eqref{eq:BPP} in situations where $S$ is not single-valued.
\begin{example}\label{ex:variational_analysis_approach_ii}
	Consider the bilevel optimization problem
	\[
		\min\limits_{x,y}\{-x-y\,|\,-x+y\leq 1,\,y\in S(x)\}
	\]
	where $S\colon\R\rightrightarrows\R$ is the solution map of $\min_y\{xy^2\,|\,y\in[-1,1]\}$.
	Here, we have
	\[
		\forall x\in\R\colon\quad
		S(x)=
			\begin{cases}
				\{-1,1\} 	&x<0,\\
				[-1,1]		&x=0,\\
				\{0\}		&x>0.
			\end{cases}
	\]
	We consider $(\bar x,\bar y):=(0,1)$.
	The associated system \eqref{eq:first_order_sufficient_condition_system_abstract_F},
	\eqref{eq:first_order_sufficient_condition_system_abstract_G},
	\eqref{eq:first_order_sufficient_VA} reduces to
	\[
		\begin{aligned}
			-\delta_x-\delta_y\leq 0,	\quad
			-\delta_x+\delta_y\leq 0,	\quad
			2\delta_x+\mu =0,			\quad
			\mu\geq 0,					\quad
			\mu\delta_y=0,				\quad
			\delta_y\leq 0,
		\end{aligned}
	\]
	and one can easily check that all of its solutions satisfy $\delta_x=\delta_y=0$ which
	means that $(\bar x,\bar y)$ is a strict local minimizer of the considered
	bilevel optimization problem, see \cref{thm:first_order_sufficient_condition_var_analysis}.
\end{example}

\subsection{KKT-approach}

In this section, we want to exploit the lower level KKT-conditions for a detailed expression of the set $M$.
This is a reasonable approach as long as the lower level problem is convex w.r.t.\ $y$ and sufficiently
regular. Below, we specify these assumptions which have to hold throughout this section.
\begin{assumption}\label{ass:KKT_approach}
	For each $x\in\R^n$, $f(x,\cdot)\colon\R^m\to\R$ is convex while $g(x,\cdot)\colon\R^m\to\R^q$ is
	componentwise convex. Furthermore, at each point of $\gph K$,
	a lower level constraint qualification holds.
\end{assumption}

Due to the postulated convexity and regularity assumptions, writing $(x,y)\in\gph S$ is equivalent to
$(x,y)\in\gph K$ and $\Lambda(x,y)\neq\varnothing$.
For subsequent use, we introduce the classical index sets
\[
	\begin{aligned}
		\bar I^{0-}=I^{0-}(\bar x,\bar y,\bar\lambda)
		&:=
		\{i\in\{1,\ldots,q\}\,|\,\bar\lambda_i=0\,\land g_i(\bar x,\bar y)<0\},\\
		\bar I^{+0}=I^{+0}(\bar x,\bar y,\bar\lambda)
		&:=
		\{i\in\{1,\ldots,q\}\,|\,\bar\lambda_i>0\,\land g_i(\bar x,\bar y)=0\},\\
		\bar I^{00}=I^{00}(\bar x,\bar y,\bar\lambda)
		&:=
		\{i\in\{1,\ldots,q\}\,|\,\bar\lambda_i=0\,\land g_i(\bar x,\bar y)=0\},
	\end{aligned}
\]
where $(\bar x,\bar y)\in\gph S$ and $\bar\lambda\in\Lambda(\bar x,\bar y)$ are
arbitrarily chosen. Clearly, $\bar I^{+0}\cup \bar I^{00}=\bar I^g$ holds by definition.

Due to the results from \cite{DempeDutta2012}, we want to abstain from using the KKT-reformulation of
the bilevel programming problem \eqref{eq:BPP} since the surrogate \emph{mathematical program with
complementarity constraints} (MPCC for short) given by
\begin{equation}\label{eq:KKT_reformulation}\tag{KKT}
	\min\limits_{x,y,\lambda}\{F(x,y)\,|\,G(x,y)\leq 0,\,\nabla_yL(x,y,\lambda)=0,\,0\leq\lambda\perp g(x,y)\leq 0\},
\end{equation}
where the lower level Lagrange multiplier is treated as an explicit variable,
may possess artificial local minimizers
which are not related to local minimizers of \eqref{eq:BPP}. Additionally, one has to check local optimality
w.r.t.\ \emph{all} associated lower level Lagrange multipliers for \eqref{eq:KKT_reformulation} in order to verify local
optimality of the underlying feasible point of \eqref{eq:BPP}.
For later use, however, let $M_\textup{KKT}\subseteq\R^n\times\R^m\times\R^q$ be the feasible set of
\eqref{eq:KKT_reformulation}. Then, we clearly have the relation $M=\Pi M_\textup{KKT}$, 
where $\Pi\in\R^{(n+m)\times(n+m+q)}$ is the projection matrix 
\[
	\Pi:=
	\begin{bmatrix}
		I_n			&	\mathbb O	&	\mathbb O\\
		\mathbb O	&	I_m			&	\mathbb O
	\end{bmatrix}.
\]
Now, \cref{lem:direct_image} opens a way to compute an upper approximate of the tangent cone to $M$.
\begin{lemma}\label{lem:tangent_cone_via_KKT_approach}
	Let $(\bar x,\bar y)\in\R^n\times\R^m$ be a feasible point of \eqref{eq:BPP} and
	let $\bar\lambda\in\Lambda(\bar x,\bar y)$ be chosen such that
	the condition
	\begin{equation}\label{eq:SMFC}
		\left.
		\begin{aligned}
			&0=\nabla _yg(\bar x,\bar y)^\top\mu,\\
			&\forall i\notin \bar I^g\colon\,\mu_i=0,\\
			&\forall i\in \bar I^{00}\colon\,\mu_i\geq 0
		\end{aligned}
		\right\}
		\,\Longrightarrow\,
		\mu=0
	\end{equation}
	holds true. Then, we have
	\[
		\mathcal T_M(\bar x,\bar y)
		\subseteq
		\Pi\left\{
			(\delta_x,\delta_y,\delta_\lambda)\,\middle|\,
				\begin{aligned}
					\nabla _xG_i(\bar x,\bar y)^\top\delta_x+\nabla_yG_i(\bar x,\bar y)^\top\delta_y&\,\leq\,0&&i\in \bar I^G\\
					\nabla^2_{yx}L(\bar x,\bar y,\bar\lambda)\delta_x+\nabla^2_{yy}L(\bar x,\bar y,\bar\lambda)\delta_y
					+\nabla_yg(\bar x,\bar y)^\top\delta_\lambda&\,=\,0&&\\
					(\delta_\lambda)_i&\,=\,0&&i\in \bar I^{0-}\\
					\nabla_xg_i(\bar x,\bar y)^\top\delta_x+\nabla_yg_i(\bar x,\bar y)^\top\delta_y&\,=\,0&&i\in \bar I^{+0}\\
					0\,\leq\,(\delta_\lambda)_i\perp\nabla_xg_i(\bar x,\bar y)^\top\delta_x+\nabla_yg_i(\bar x,\bar y)^\top\delta_y
					&\,\leq\,0& &i\in \bar I^{00}
				\end{aligned}
		\right\}.
	\]
\end{lemma}
\begin{proof}
	Due to \cref{lem:direct_image}, we have
	\[
		\mathcal T_M(\bar x,\bar y)
		=
		\mathcal T_{\Pi M_\textup{KKT}}(\bar x,\bar y)
		=
		\cl\left(\Pi\mathcal T_{M_\textup{KKT}}(\bar x,\bar y,\bar\lambda)\right)
	\]
	provided that the constraint qualification
	\begin{equation}\label{eq:direct_image_CQ}
		\operatorname{ker}\Pi\cap\mathcal T_{M_\textup{KKT}}(\bar x,\bar y,\bar\lambda)=\{0\}
	\end{equation}
	holds true. Obviously, $\operatorname{ker}\Pi=\{0\}\times\{0\}\times\R^q$ is valid.
	It is well known from \cite[Lemma~3.2]{FlegelKanzow2005} that the so-called MPCC-tailored
	linearization cone
	\[
		\mathcal L_{M_\textup{KKT}}(\bar x,\bar y,\bar\lambda)
		:=
		\left\{
			(\delta_x,\delta_y,\delta_\lambda)\,\middle|\,
				\begin{aligned}
					\nabla _xG_i(\bar x,\bar y)^\top\delta_x+\nabla_yG_i(\bar x,\bar y)^\top\delta_y&\,\leq\,0&&i\in \bar I^G\\
					\nabla^2_{yx}L(\bar x,\bar y,\bar\lambda)\delta_x+\nabla^2_{yy}L(\bar x,\bar y,\bar\lambda)\delta_y
					+\nabla_yg(\bar x,\bar y)^\top\delta_\lambda&\,=\,0&&\\
					(\delta_\lambda)_i&\,=\,0&&i\in \bar I^{0-}\\
					\nabla_xg_i(\bar x,\bar y)^\top\delta_x+\nabla_yg_i(\bar x,\bar y)^\top\delta_y&\,=\,0&&i\in \bar I^{+0}\\
					0\,\leq\,(\delta_\lambda)_i\perp\nabla_xg_i(\bar x,\bar y)^\top\delta_x+\nabla_yg_i(\bar x,\bar y)^\top\delta_y
					&\,\leq\,0& &i\in \bar I^{00}
				\end{aligned}
		\right\}
	\]
	always provides an upper estimate of $\mathcal T_{M_{\textup{KKT}}}(\bar x,\bar y,\bar\lambda)$.
	
	Let us note that any vector
	$(\delta_x,\delta_y,\delta_\lambda)\in\operatorname{ker}\Pi\cap\mathcal L_{M_\textup{KKT}}(\bar x,\bar y,\bar\lambda)$,
	by definition,
	satisfies $\delta_x=\delta_y=0$, $\nabla_yg(\bar x,\bar y)^\top\delta_\lambda=0$, $(\delta_\lambda)_i=0$ for all
	$i\in\bar I^{0-}=\{1,\ldots,q\}\setminus\bar I^g$, and $(\delta_\lambda)_i\geq 0$ for all $i\in \bar I^{00}$.
	Now, the validity of \eqref{eq:SMFC} guarantees $\delta_\lambda=0$.
	Particularly, \eqref{eq:direct_image_CQ} holds as well.
	As a consequence, we have
	\[
		\mathcal T_M(\bar x,\bar y)
		=
		\cl\left(\Pi\mathcal T_{M_\textup{KKT}}(\bar x,\bar y,\bar\lambda)\right)
		\subseteq
		\cl\left(\Pi\mathcal L_{M_\textup{KKT}}(\bar x,\bar y,\bar\lambda)\right)
		=
		\Pi\mathcal L_{M_\textup{KKT}}(\bar x,\bar y,\bar\lambda)
	\]
	since $\mathcal L_{M_\textup{KKT}}(\bar x,\bar y,\bar\lambda)$ is the union of finitely many
	polyhedral cones which implies that $\Pi\mathcal L_{M_\textup{KKT}}(\bar x,\bar y,\bar\lambda)$
	possesses the same property and, thus, is closed.
	This shows the desired estimate.
\end{proof}

Let us comment on the above lemma in the subsequent remarks.
\begin{remark}\label{rem:SMFC}
	In the literature, see \cite{Kyparisis1985,Wachsmuth2013}, condition
	\eqref{eq:SMFC}
	is referred to as \emph{Strict Mangasarian--Fromovitz Condition}.
	It is well known that this condition, which cannot be called a constraint qualification
	since it depends explicitly on a Lagrange multiplier $\bar\lambda$ and, thus, implicitly on the lower level
	objective functional, already implies the uniqueness of the multiplier, i.e.\ $\Lambda(\bar x,\bar y)$
	reduces to the singleton $\{\bar\lambda\}$.
	It can be easily checked that validity of \eqref{eq:SMFC} is inherent whenever LLICQ holds
	at the reference point, but \eqref{eq:SMFC} may hold in situations where
	LLICQ is violated, see \cite[Example~2.2]{Kyparisis1985}.
\end{remark}
\begin{remark}\label{rem:sharp_estimate_under_MPCC_ACQ}
	Following the lines of the proof of \cref{lem:tangent_cone_via_KKT_approach}, the provided
	estimate for $\mathcal T_M(\bar x,\bar y)$ is exact whenever the relation
	$\mathcal T_{M_\textup{KKT}}(\bar x,\bar y,\bar\lambda)=\mathcal L_{M_\textup{KKT}}(\bar x,\bar y,\bar\lambda)$
	is valid. In the literature on MPCCs, this condition is refereed to as MPCC-ACQ, see,
	e.g., \cite[Definition~3.1]{FlegelKanzow2005}, and is a comparatively weak MPCC-tailored
	constraint qualification. It holds, e.g., when $G$ is affine and the lower level problem
	is of obstacle-type \eqref{eq:obstacle_problem}, see \cite[Theorem~3.2]{FlegelKanzow2005}.
	In this setting, LLICQ is valid at all lower level feasible points and, thus, \eqref{eq:SMFC}
	is valid as well, i.e.\ \cref{lem:tangent_cone_via_KKT_approach} provides a precise formula
	for the tangent cone to $M$ at $(\bar x,\bar y)$.
\end{remark}

Combining \cref{lem:trivial_first_order_condition,lem:tangent_cone_via_KKT_approach}, we have the following
first-order sufficient optimality condition.
\begin{theorem}\label{thm:first_order_sufficient_condition_KKT}
	Let $(\bar x,\bar y)\in\R^n\times\R^m$ be a feasible point of \eqref{eq:BPP}.
	Furthermore, fix $\bar\lambda\in\Lambda(\bar x,\bar y)$ satisfying \eqref{eq:SMFC}.
	Finally, suppose that the system \eqref{eq:first_order_sufficient_condition_system_abstract_F},
	\eqref{eq:first_order_sufficient_condition_system_abstract_G},
	\begin{subequations}\label{eq:first_order_sufficient_KKT}
		\begin{align}
			\label{eq:first_order_sufficient_KKT_Lagrangian}
				\nabla^2_{yx}L(\bar x,\bar y,\bar\lambda)\delta_x
				+
				\nabla^2_{yy}L(\bar x,\bar y,\bar\lambda)\delta_y
				+
				\nabla _yg(\bar x,\bar y)^\top\delta_\lambda
				&\,=\,0,\\
			\label{eq:first_order_sufficient_KKT_I0-}
				(\delta_\lambda)_i&\,=\,0\quad i\in\bar I^{0-},\\
			\label{eq:first_order_sufficient_KKT_I+0}
				\nabla_xg_i(\bar x,\bar y)^\top\delta_x+\nabla_yg_i(\bar x,\bar y)^\top\delta_y&\,=\,0
				\quad i\in \bar I^{+0},\\
			\label{eq:first_order_sufficient_KKT_I00}
				0\,\leq\,(\delta_\lambda)_i\,\perp\,
				\nabla_xg_i(\bar x,\bar y)^\top\delta_x+\nabla_yg_i(\bar x,\bar y)^\top\delta_y
				&\,\leq\,0\quad i\in\bar I^{00}
		\end{align}
	\end{subequations}
	does not possess a nontrivial solution $(\delta_x,\delta_y,\delta_\lambda)$.
	Then, $(\bar x,\bar y)$ is a strict local minimizer of \eqref{eq:BPP}.
\end{theorem}

In the subsequent remarks, we comment on this theorem and its assumptions.
\begin{remark}\label{rem:multiplier_component}
	Under validity of \eqref{eq:SMFC}, each solution $(\delta_x,\delta_y,\delta_\lambda)$ of the system
	\eqref{eq:first_order_sufficient_condition_system_abstract_F},
	\eqref{eq:first_order_sufficient_condition_system_abstract_G},
	\eqref{eq:first_order_sufficient_KKT} which satisfies
	$(\delta_x,\delta_y)=(0,0)$ already satisfies $\delta_\lambda=0$.
	Thus, one could equivalently state the assertion of \cref{thm:first_order_sufficient_condition_KKT}
	via postulating that the system
	\eqref{eq:first_order_sufficient_condition_system_abstract_F},
	\eqref{eq:first_order_sufficient_condition_system_abstract_G},
	\eqref{eq:first_order_sufficient_KKT}
	does not possess a solution $(\delta_x,\delta_y,\delta_\lambda)$ which satisfies
	$(\delta_x,\delta_y)\neq(0,0)$,
	and the latter seems to be natural in light of $\mathcal T_M(\bar x,\bar y)\subseteq\R^n\times\R^m$
	and \cref{lem:trivial_first_order_condition}.
\end{remark}
\begin{remark}\label{rem:cosequences_for_KKT}
	Exploiting \cite[Lemma~3.2]{FlegelKanzow2005}, one can show that the assumptions of
	\cref{thm:first_order_sufficient_condition_KKT} guarantee that $(\bar x,\bar y,\bar\lambda)$
	is a strict local minimizer of \eqref{eq:KKT_reformulation}.
	Observing that the validity of \eqref{eq:SMFC} implies that $\Lambda(\bar x,\bar y)=\{\bar\lambda\}$
	holds, see \cref{rem:SMFC}, we can use \cite[Theorem~3.2]{DempeDutta2012} in order to
	infer that $(\bar x,\bar y)$ is a strict local minimizer of \eqref{eq:BPP}.
	This provides another proof strategy for \cref{thm:first_order_sufficient_condition_KKT}.
\end{remark}
\begin{remark}\label{rem:non_unique_multipliers}
	Let us note that for each local minimizer $(\tilde x,\tilde y)$ of \eqref{eq:BPP}
	such that $\Lambda(\tilde x,\tilde y)$ is not a singleton, the point $(\tilde x,\tilde y,\tilde \lambda)$
	cannot be a \emph{strict} local minimizer of \eqref{eq:KKT_reformulation} for each
	$\tilde\lambda\in\Lambda(\tilde x,\tilde y)$ since the objective functional of
	\eqref{eq:KKT_reformulation} does not depend on $\lambda$.
	Thus, the KKT-approach is not suited for identifying strict local minimizer of \eqref{eq:BPP}
	with nonunique lower level multiplier since first- (or second-) order sufficient optimality
	conditions imply local linear (or quadratic) growth of the objective functional
	around the point of interest.
	In this regard, the restriction to situations where the lower level Lagrange multiplier is
	uniquely determined seems to be quite natural in the context of this section.
	Noting that the Strict Mangasarian--Fromovitz Condition \eqref{eq:SMFC} is the weakest
	condition which guarantees this, see \cite[Proposition~1.1]{Kyparisis1985},
	the assumptions of \cref{thm:first_order_sufficient_condition_KKT} are quite reasonable.
\end{remark}
\begin{example}\label{ex:first_order_sufficient_conditions_KKT_approach}
	We consider the bilevel programming problem from \cref{ex:variational_analysis_approach_ii}
	at $(\bar x,\bar y):=(0,1)$. The associated lower level Lagrange multiplier is uniquely
	determined and given by $\bar\lambda=(0,0)$. This yields $\bar I^{0-}=\{1\}$ and $\bar I^{00}=\{2\}$
	where we used $g_1(y):=-y-1$ and $g_2(y):=y-1$ for all $y\in\R$.
	The associated system \eqref{eq:first_order_sufficient_condition_system_abstract_F},
	\eqref{eq:first_order_sufficient_condition_system_abstract_G},
	\eqref{eq:first_order_sufficient_KKT} reads as
	\[
		\begin{aligned}
			-\delta_x-\delta_y\leq 0,	\quad
			\delta_x-\delta_y\leq 0,	\quad
			2\delta_x-(\delta_\lambda)_1+(\delta_\lambda)_2=0,	\quad
			(\delta_\lambda)_1=0,	\quad
			0\leq(\delta_\lambda)_2\perp\delta_y\leq 0.
		\end{aligned}
	\]
	and, obviously, does not possess a nontrivial solution.
	Thus, \cref{thm:first_order_sufficient_condition_KKT} correctly identifies the strict local minimizer
 $(\bar x,\bar y)$.
\end{example}

\section{A second-order sufficient optimality condition}\label{sec:second_order_optimality_conditions}

In this section, we are going to derive a second-order sufficient optimality condition for \eqref{eq:BPP}
which is based on \eqref{eq:LLVF}. Therefore, we will make use of the second-order directional derivative
of the optimal value function $\varphi$. First, we are going to state a rather general result which is,
afterwards, specified to different settings where the assumptions are partially inherent or can be easily
verified.

Let $(\bar x,\bar y)\in\R^n\times\R^m$ be a feasible point of \eqref{eq:BPP} and assume that $\varphi$ is
directionally differentiable at $\bar x$.
For subsequent use, we introduce the linearization cone
\begin{equation}\label{eq:Linearization cone}
\mathcal L^\varphi_M(\bar x,\bar y)
	:=
	\left\{
		d:=(\delta_x,\delta_y)\in\R^n\times\R^m\,\middle|\,
			\begin{aligned}
				\nabla G_i(\bar x,\bar y)^\top d&\,\leq\,0&&i\in \bar I^G\\
				\nabla f(\bar x,\bar y)^\top d-\varphi'(\bar x;\delta_x)&\,\leq\,0&&\\
				\nabla g_i(\bar x,\bar y)^\top d&\,\leq\,0&&i\in \bar I^g
			\end{aligned}
	\right\}
\end{equation}
and the critical cone
\[
	\mathcal C^\varphi_M(\bar x,\bar y)
	:=
	\mathcal L^\varphi_M(\bar x,\bar y)\cap\{d\in\R^n\times\R^m\,|\,
	\nabla F(\bar x,\bar y)^\top d\,\leq\,0
	\},
\]
of \eqref{eq:LLVF} at $(\bar x,\bar y)$.
Note that $\mathcal L^\varphi_M(\bar x,\bar y)$ and $\mathcal C^\varphi_M(\bar x,\bar y)$ do not need to be convex
as soon as the function $\varphi$ is nonsmooth at the point $\bar x$.

\begin{remark}\label{rem:no_decrease_in_optimal_value_constraint}
	Let $(\bar x,\bar y)\in\R^n\times\R^m$ be a feasible point of \eqref{eq:BPP}
	and let $\varphi$ be locally Lipschitz continuous and directionally differentiable at $\bar x$.
	Furthermore, assume that at least ACQ holds for the set
	\begin{equation}\label{eq:points_satisfying_lower_level_constraints}
		\{(x,y)\in\R^n\times\R^m\,|\,g(x,y)\leq 0\}
	\end{equation}
	at $(\bar x,\bar y)$.
	Noting that $(\bar x,\bar y)$ is a global minimizer of
	\[
		\min\limits_{x,y}\{f(x,y)-\varphi(x)\,|\,g(x,y)\leq 0\}
	\]	
	by definition of $\varphi$, \cref{lem:abstract_first_order_necessary_optimality_condition}
	and the validity of ACQ yield
	\begin{equation}\label{eq:no_descent_direction_OVC}
		\forall d:=(\delta_x,\delta_y)\in\R^n\times\R^m\colon\quad
		(\forall i\in\bar I^g\colon\;\nabla g_i(\bar x,\bar y)^\top d\leq 0)\,\Longrightarrow\,
		\nabla f(\bar x,\bar y)^\top d-\varphi'(\bar x;\delta_x)\geq 0.
	\end{equation}
	Thus, in this situation, the linearization cone $\mathcal L^\varphi_M(\bar x,\bar y)$ possesses
	the representation
	\[
	\mathcal L^\varphi_M(\bar x,\bar y)
	=
	\left\{
		d:=(\delta_x,\delta_y)\in\R^n\times\R^m\,\middle|\,
			\begin{aligned}
				\nabla G_i(\bar x,\bar y)^\top d&\,\leq\,0&&i\in \bar I^G\\
				\nabla f(\bar x,\bar y)^\top d-\varphi'(\bar x;\delta_x)&\,=\,0&&\\
				\nabla g_i(\bar x,\bar y)^\top d&\,\leq\,0&&i\in \bar I^g
			\end{aligned}
	\right\}.
	\]
\end{remark}

The upcoming example depicts that the statement of \cref{rem:no_decrease_in_optimal_value_constraint}
does not generally hold in the absence of a constraint qualification.
\begin{example}\label{ex:decrease_in_optimal_value_constraint_without_ACQ}
	Let us consider the lower level problem
	\[
		S(x):=\argmin\limits_y\{y_1\,|\,y_1^2+y_2\leq 0,\,y_1^2+(y_2-1)^2\leq 1+x^2\}.
	\]
	One can check that it holds
	\[
		\forall x\in\R\colon\quad
		S(x)
		=
		\left\{
			\left(-\sqrt{\sqrt{2.25+x^2}-1.5},1.5-\sqrt{2.25+x^2}\right)
		\right\}.
	\]
	Thus, the optimal value function $\varphi\colon\R\to\R$ is given by
	\[
		\forall x\in\R\colon\quad
		\varphi(x)
		=
		-\sqrt{\sqrt{2.25+x^2}-1.5}
		=
		-\frac{|x|}{\sqrt{\sqrt{2.25+x^2}+1.5}}.
	\]
	Let us consider $(\bar x,\bar y):=(0,0)\in\gph S$. At $\bar x$, $\varphi$ is
	locally Lipschitz continuous as well as directionally differentiable, and it holds
	\[
		\forall\delta_x\in\R\colon\quad
		\varphi'(\bar x;\delta_x)
		=
		-\tfrac{1}{\sqrt 3}|\delta_x|.
	\]
	The linearization cone to the set given in \eqref{eq:points_satisfying_lower_level_constraints}
	at $(\bar x,\bar y)$ is given by $\{(\delta_x,\delta_{y,1},\delta_{y,2})\in\R^3\,|\,\delta_{y,2}=0\}$.
	Particularly, the direction $(0,-1,0)$ belongs to this cone although it is not a tangent direction,
	i.e.\ ACQ does not hold at $(\bar x,\bar y)$ for the set in \eqref{eq:points_satisfying_lower_level_constraints}.
	Using this direction, one can check that \eqref{eq:no_descent_direction_OVC} is clearly
	violated, and, consequently, the potential representation of the linearization cone from
	\cref{rem:no_decrease_in_optimal_value_constraint}
	is strictly smaller than the original linearization cone from \eqref{eq:Linearization cone}.
\end{example}

Now, we are in position to state the second-order sufficient condition of interest.
Our arguments here are inspired by \cite[Theorem~4.1]{RueckmannShapiro2001}.
\begin{theorem}\label{thm:abstract_second_order_condition}
	Let $(\bar x,\bar y)\in\R^n\times\R^m$ be feasible to \eqref{eq:BPP} such that at least ACQ holds for
	the set in \eqref{eq:points_satisfying_lower_level_constraints} at $(\bar x,\bar y)$.
	Let $\varphi$ be locally Lipschitzian and second-order directionally differentiable at $\bar x$.
	Moreover, let $-\varphi$ be second-order epi-regular at	$\bar x$.
	Finally, assume that we have
	\begin{equation}\label{eq:no_descent_direction}
		\forall d\in\mathcal L^\varphi_M(\bar x,\bar y)\colon\quad
		\nabla F(\bar x,\bar y)^\top d\geq 0,
	\end{equation}
	and that the optimization problem
	\begin{equation}\label{eq:second_order_sufficient_condition_subproblem}
		\min\limits_{w=(\omega_x,\omega_y),\alpha}
		\left\{\alpha\,\middle|\,
			\begin{aligned}
			\nabla F(\bar x,\bar y)^\top w+d^\top\nabla^2F(\bar x,\bar y)d&\,\leq\,\alpha&&\\
			\nabla G_i(\bar x,\bar y)^\top w+d^\top\nabla^2G_i(\bar x,\bar y)d&\,\leq\,\alpha&&i\in \bar I^G(d)\\
			\nabla f(\bar x,\bar y)^\top w+d^\top\nabla^2f(\bar x,\bar y)d-\varphi''(\bar x;\delta_x,\omega_x)&\,\leq\,\alpha&&\\
			\nabla g_i(\bar x,\bar y)^\top w+d^\top\nabla^2g_i(\bar x,\bar y)d&\,\leq\,\alpha&&i\in \bar I^g(d)
		\end{aligned}
		\right\}
	\end{equation}
	possesses a positive optimal objective value for each $d\in\mathcal C^\varphi_M(\bar x,\bar y)\setminus\{0\}$. Here, we use
	\[
		\bar I^G(d) := \{i\in\bar I^G\,|\,\nabla G_i(\bar x,\bar y)^\top d=0\} \qquad\mbox{ and }\qquad
		\bar I^g(d) := \{i\in\bar I^g\,|\,\nabla g_i(\bar x,\bar y)^\top d=0\}.
	\]
	Then, $(\bar x,\bar y)$ is a strict local minimizer of \eqref{eq:BPP} which satisfies the second-order
	growth condition.
\end{theorem}
\begin{proof}
	For the proof, we exploit the functional $\psi\colon\R^n\times\R^m\to\overline\R$ defined, for all $(x,y)\in\R^n\times\R^m$, by
	\[
	\psi(x,y):=\max\{F(x,y)-F(\bar x,\bar y),G_1(x,y),\ldots,G_p(x,y),f(x,y)-\varphi(x),g_1(x,y),\ldots,g_q(x,y)\}.
	\]
	
	As a first step of the proof, we show that whenever the second-order growth condition holds for
	$\psi$ at $(\bar x,\bar y)$, then it also holds for the optimization problem \eqref{eq:LLVF} at $(\bar x,\bar y)$.
	Thus, we assume that there are constants $\varepsilon>0$ and $c>0$ such that
	\begin{equation}\label{eq:second_order_growth_localization}
		\forall (x,y)\in\mathbb B_\varepsilon(\bar x,\bar y)\colon\quad
		\psi(x,y)\geq c\left(\norm{x-\bar x}{2}^2+\norm{y-\bar y}{2}^2\right)
	\end{equation}
	holds. Now, suppose that the second-order growth condition does not hold for \eqref{eq:LLVF} at $(\bar x,\bar y)$.
	Then, we find a sequence $\{(x_k,y_k)\}_{k\in\N}\subseteq M$ converging to $(\bar x,\bar y)$ which satisfies
	\[
		\forall k\in\N\colon\quad
		F(x_k,y_k)<F(\bar x,\bar y)+\tfrac1k\left(\norm{x_k-\bar x}{2}^2+\norm{y_k-\bar y}{2}^2\right).
	\]
	By definition of $\psi$, this yields
	\[
		\forall k\in\N\colon\quad
		\psi(x_k,y_k)
			\leq\max\{F(x_k,y_k)-F(\bar x,\bar y),0\}
			<\tfrac1k\left(\norm{x_k-\bar x}{2}^2+\norm{y_k-\bar y}{2}^2\right);
	\]
	clearly contradicting \eqref{eq:second_order_growth_localization}.
	Thus, the second-order growth condition needs to hold for \eqref{eq:LLVF} at $(\bar x,\bar y)$.
		
	In the remaining part of the proof, we verify that $\psi$ indeed satisfies the second-order
	growth condition at $(\bar x,\bar y)$ which, due to the above arguments, would yield the claim.
	This can be done with the aid of \cref{prop:second_order_sufficient_condition_unconstrained}.
	Due to the assumptions of the theorem, we note that $\psi$ is locally Lipschitz continuous and second-order
	directionally differentiable at $(\bar x,\bar y)$, see \cref{lem:max_rule}.
	Since $-\varphi$ is assumed to be second-order epi-regular at $\bar x$, we can exploit
	\cref{lem:second_order_epiregularity_max} in order to see that $\psi$ is second-order
	epi-regular at $(\bar x,\bar y)$.
	Noting that we have
	\[
		\begin{aligned}
		\psi'((\bar x,\bar y);d)
		=
		\max\Bigl\{
			\nabla F(\bar x,\bar y)^\top d,
			&\max\left\{\nabla G_i(\bar x,\bar y)^\top d\,|\,i\in\bar I^G\right\},\\
			&\nabla f(\bar x,\bar y)^\top d-\varphi'(\bar x;\delta_x),
			\max\left\{\nabla g_i(\bar x,\bar y)^\top d\,|\,i\in\bar I^g\right\}
		\Bigr\}		
		\end{aligned}			
	\]
	for all $d:=(\delta_x,\delta_y)\in\R^n\times\R^m$ from \cref{lem:max_rule} and the equality
	$f(\bar x,\bar y)-\varphi(\bar x)=0$ which holds by definition of $\varphi$,
	condition \eqref{eq:no_descent_direction} implies
	$\psi'((\bar x,\bar y);d)\geq 0$ for all $d\in\R^n\times\R^m$.
	Next, we observe that due to \eqref{eq:no_descent_direction}, it holds
	\[
		\mathcal C^\varphi_M(\bar x,\bar y)=\{d\in\R^n\times\R^m\,|\,\psi'((\bar x,\bar y);d)=0\}.
	\]
	Pick $d:=(\delta_x,\delta_y)\in\mathcal C^\varphi_M(\bar x,\bar y)$ arbitrarily.
	Due to \eqref{eq:no_descent_direction}, it holds $\nabla F(\bar x,\bar y)^\top d=0$.
	Furthermore, \cref{rem:no_decrease_in_optimal_value_constraint} guarantees
	$\nabla f(\bar x,\bar y)^\top d-\varphi'(\bar x;\delta_x)=0$.	
	Exploiting \cref{lem:max_rule} again, we obtain
	\[
		\begin{aligned}
		\psi''((\bar x,\bar y);d,w)
		=
		\max\Bigl\{
			&\nabla F(\bar x,\bar y)^\top w+d^\top\nabla^2F(\bar x,\bar y)d,\\
			&\max
				\left\{
					\nabla G_i(\bar x,\bar y)^\top w+d^\top\nabla^2G_i(\bar x,\bar y)d
					\,\middle|\,
					i\in\bar I^G(d)
				\right\},\\
			&\nabla f(\bar x,\bar y)^\top w+d^\top\nabla^2f(\bar x,\bar y)d-\varphi''(\bar x;\delta_x,\omega_x),\\
			&\max
				\left\{
					\nabla g_i(\bar x,\bar y)^\top w+d^\top\nabla^2g_i(\bar x,\bar y)d
					\,\middle|\,
					i\in \bar I^g(d)
				\right\}
			\Bigr\}
		\end{aligned}
	\]
	for all $w=(\omega_x,\omega_y)\in\R^n\times\R^m$.
	As a consequence, the assumptions of the proposition guarantee that
	\[
		\forall d\in\mathcal C^\varphi_M(\bar x,\bar y)\setminus\{0\}\colon\quad
		\inf_{w\in\R^n\times\R^m}\psi''((\bar x,\bar y);d,w)>0
	\]
	holds true.	
	Due to the above observations, \cref{prop:second_order_sufficient_condition_unconstrained} yields that
	the second-order growth condition holds for $\psi$ at $(\bar x,\bar y)$, and due to the first part of
	the proof, the claim of the theorem holds.
\end{proof}

Let us interrelate the second-order sufficient condition of \cref{thm:abstract_second_order_condition} with
other second-order sufficient optimality conditions from the literature of nonlinear optimization. Therefore,
we suppose for a moment that $\varphi$ is twice continuously differentiable (in particular, $-\varphi$ is
second-order epi-regular in this setting).
This is guaranteed under validity of LLICQ, LSOSC, and LSC at the reference point
from $\gph S$ provided that the lower level problem is convex for each $x\in\R^n$, and in this particular
situation, the associated Hessian of $\varphi$ can be computed in terms of
initial problem data, see \cite[Theorem~3.4.1]{Fiacco1983} for details.
Exploiting the Farkas lemma, one can easily check that \eqref{eq:no_descent_direction} equals the 
KKT-conditions of \eqref{eq:LLVF} since $\varphi$ is continuously differentiable at the reference parameter.
Furthermore, for fixed $d\in\mathcal C^\varphi_M(\bar x,\bar y)$,
\eqref{eq:second_order_sufficient_condition_subproblem} is a linear program whose dual is given by
\[
	\max\left\{
		d^\top\nabla^2 L^\textup{vf}(\bar x,\bar y,\sigma,\nu)d
		\,\middle|\,
		(\sigma,\nu)\in\Lambda^\textup{vf}_0(\bar x,\bar y,d)
	\right\}
\]
where $L^\textup{vf}\colon\R^n\times\R^m\times\R\times\R^{p+1+q}\to\R$ is the \emph{generalized} Lagrangian
of \eqref{eq:LLVF} given by
\[
	L^\textup{vf}(x,y,\sigma,\nu)
	:=
	\sigma\,F(x,y)+G(x,y)^\top\nu^G+\nu^\textup{vf}(f(x,y)-\varphi(x))+g(x,y)^\top\nu^g
\]
and, for each $d\in\mathcal C^\varphi_M(\bar x,\bar y)$,
$\Lambda^\textup{vf}_0(\bar x,\bar y,d)$ is the set of Fritz--John multipliers defined below:
\[
	\Lambda^\textup{vf}_0(\bar x,\bar y,d)
	:=
	\left\{
		(\sigma,\nu)\in\R\times\R^{p+1+q}
		\,\middle|\,
		\begin{aligned}
			&\nabla L^\textup{vf}(\bar x,\bar y,\sigma,\nu)=0,\\
			&\sigma\geq 0,\,\nu\geq 0,\\
			&(\sigma,\nu)\neq 0,\\
			&\forall i\notin \bar I^G(d)\colon\,\nu^G_i=0,\\
			&\forall i\notin \bar I^g(d)\colon\,\nu^g_i=0
		\end{aligned}
	\right\}.
\]
We emphasize that $\nabla L^\textup{vf}$ and $\nabla^2L^\textup{vf}$ denote the respective derivatives 
of the Lagrangian function $L^\textup{vf}$ w.r.t.\ the decision variables $x$ and $y$.
Above and subsequently, we identify $\nu\in\R^{p+1+q}$ with the triplet $(\nu^G,\nu^\textup{vf},\nu^g)\in\R^p\times\R\times\R^q$
where $\nu^G$ is the vector of the first $p$ components of $\nu$, $\nu^\textup{vf}$ is the $(p+1)$-th component of $\nu$ and
$\nu^g$ is the vector of the last $q$ components of $\nu$.
Using the above notation, the set
\[
	\Lambda^\textup{vf}(\bar x,\bar y)
	:=
	\left\{\nu\in\R^{p+1+q}\,
	\,\middle|\,
	(1,\nu)\in\Lambda^\textup{vf}_0(\bar x,\bar y,0)
	\right\}
\]
comprises the Lagrange multipliers for \eqref{eq:LLVF} at $(\bar x,\bar y)$ and is, due to \eqref{eq:no_descent_direction},
nonempty. For fixed multiplier $\bar\nu\in\Lambda^\textup{vf}(\bar x,\bar y)$, we
easily obtain the following representation of the critical cone from the KKT-system of \eqref{eq:LLVF}
and \cref{rem:no_decrease_in_optimal_value_constraint}:
\[
	\mathcal C^\varphi_M(\bar x,\bar y)
	=
	\left\{
		d:=(\delta_x,\delta_y)\in\R^n\times\R^m\,\middle|\,
			\begin{aligned}
				\nabla G_i(\bar x,\bar y)^\top d&\,\leq\,0&&i\in \bar I^G,\,\bar\nu^G_i=0\\
				\nabla G_i(\bar x,\bar y)^\top d&\,=\,0&&i\in \bar I^G,\,\bar\nu^G_i>0\\
				\nabla f(\bar x,\bar y)^\top d-\nabla\varphi(\bar x)^\top\delta_x&\,=\,0&&\\
				\nabla g_i(\bar x,\bar y)^\top d&\,\leq\,0&&i\in \bar I^g,\,\bar\nu^g_i=0\\
				\nabla g_i(\bar x,\bar y)^\top d&\,=\,0&&i\in \bar I^g,\,\bar\nu^g_i>0
			\end{aligned}
	\right\}.
\]
Thus, for each $d\in\mathcal C^\varphi_M(\bar x,\bar y)$, it naturally holds
$(1,\bar\nu)\in\Lambda^\textup{vf}_0(\bar x,\bar y,d)$, i.e.\ the latter
set is nonempty.
Coming back to \cref{thm:abstract_second_order_condition},
saying that \eqref{eq:second_order_sufficient_condition_subproblem} possesses positive objective value
for each $d\in\mathcal C^\varphi_M(\bar x,\bar y)\setminus\{0\}$ is, by strong duality of linear programming,
equivalent to postulating
\begin{equation}\label{eq:LLVF_WSOSC}
	\forall d\in\mathcal C^\varphi_M(\bar x,\bar y)\setminus\{0\}\,
	\exists (\sigma,\nu)\in\Lambda^\textup{vf}_0(\bar x,\bar y,d)\colon\quad
	d^\top\nabla^2 L^\textup{vf}(\bar x,\bar y,\sigma,\nu)d>0.
\end{equation}
Due to the above comments, this condition is implied by validity of
\begin{equation}\label{eq:LLVF_SOSC}
	\forall d\in\mathcal C^\varphi_M(\bar x,\bar y)\setminus\{0\}\,
	\exists \nu\in\Lambda^\textup{vf}(\bar x,\bar y)\colon\quad
	d^\top\nabla^2 L^\textup{vf}(\bar x,\bar y,1,\nu)d>0,
\end{equation}
and this corresponds to the classical Second-Order Sufficient Condition (SOSC) from nonlinear programming
applied to \eqref{eq:LLVF}. Clearly, \eqref{eq:LLVF_WSOSC} is weaker than \eqref{eq:LLVF_SOSC}.

We observe from above that the requirements from \cref{thm:abstract_second_order_condition}
provide a weak Fritz--John-type second-order sufficient condition in primal form which can
be dualized whenever $\varphi$ is sufficiently smooth.
\begin{example}\label{ex:second_order_sufficient_condition_varphi_locally_smooth}
	We consider the bilevel programming problem from \cref{ex:first_order_sufficient_condition}
	at $(\bar x,\bar y):=(-3,1)$. Locally around $\bar x$, $\varphi$ is affine and, thus, smooth.
	One can easily check that
	\[
		\Lambda^\textup{vf}(\bar x,\bar y)
		=
		\{
			(\nu^\textup{vf},0,\nu^g_2)\in\R^3_+\,|\,2-3\nu^\textup{vf}+\nu^g_2=0
		\}
	\]
	holds true, and this set is nonempty, i.e.\ the KKT-conditions hold for the
	associated model \eqref{eq:LLVF} at $(\bar x,\bar y)$.
	Let us fix $\bar\nu^\textup{vf}:=1$ and $\bar\nu^g_2:=1$.
	Then, it holds $(\bar\nu^\textup{vf},0,\bar\nu^g_2)\in\Lambda^\textup{vf}(\bar x,\bar y)$
	and
	\[
		\nabla^2 L^\textup{vf}(\bar x,\bar y,1,\bar\nu^\textup{vf},0,\bar\nu^g_2)
		=
		\begin{pmatrix}
			1&1\\1&1
		\end{pmatrix}.
	\]
	Furthermore, one can easily check that $\mathcal C^\varphi_M(\bar x,\bar y)=\R\times\{0\}$ is
	valid. Consequently, condition \eqref{eq:LLVF_SOSC} is satisfied, and, hence, $(\bar x,\bar y)$ is a
	strict local minimizer of the underlying bilevel optimization problem.
\end{example}

Despite the above observations, \cref{thm:abstract_second_order_condition} is of limited use in practice due
to the following observations. First, it demands several assumptions on the implicitly given function $\varphi$.
Second, the appearing linearization and critical cone depend on the directional derivative of $\varphi$. Third,
the program \eqref{eq:second_order_sufficient_condition_subproblem} involves evaluations of the second-order
directional derivative of the function $\varphi$. In order to obtain applicable conditions, we need to focus on
specific instances of \eqref{eq:BPP} where most of these issues can be discussed in terms of initial problem data.
This will be done in the upcoming subsections for three special classes of bilevel optimization problems.

Note that \eqref{eq:no_descent_direction} is a comparatively strong condition. Indeed, if
$\varphi$ is smooth at $\bar x$, then this requirement is equivalent to postulating that $(\bar x,\bar y)$
is a KKT-point of \eqref{eq:LLVF}, see \cref{prop:B_St_vs_KKT_of_LLVF} below. The class of bilevel
programming problems whose local minimizers satisfy (abstract) KKT-conditions of \eqref{eq:LLVF} is
closely related to those programs enjoying the property to be partially calm at local minimizers.
\begin{remark}\label{rem:partial_exact_penalization}
	Let all the assumptions of \cref{thm:abstract_second_order_condition} be satisfied
	at some feasible point $(\bar x,\bar y)\in\R^n\times\R^m$ of \eqref{eq:BPP}.
	Thus, $(\bar x,\bar y)$ is, in some sense, a stationary point of
	\eqref{eq:LLVF} where a (primal) Fritz--John-type second-order sufficient optimality condition
	holds, cf.\ \eqref{eq:LLVF_WSOSC} for the case where $\varphi$ is twice continuously
	differentiable.
	In light of standard results regarding penalty methods, one might ask now whether
	$(\bar x,\bar y)$ is a local minimizer of the partially penalized problem
	\eqref{eq:partially_penalized_LLVF} for some $\kappa>0$ or, equivalently, whether
	\eqref{eq:LLVF} is necessarily partially calm at $(\bar x,\bar y)$.
	Following e.g.\ \cite{GeigerKanzow2002} or \cite{HanMangasarian1979},
	this requires either some strong second-order sufficient condition to hold at $(\bar x,\bar y)$
	or an MFCQ-type constraint qualification to be valid for \eqref{eq:LLVF}
	at $(\bar x,\bar y)$. Both these requirements are, however, not reasonable
	in the setting at hand. We already observed that \cref{thm:abstract_second_order_condition}
	only provides a weak second-order condition while it is folklore that
	\eqref{eq:LLVF} is inherently irregular.
	Consequently, \cref{thm:abstract_second_order_condition} might be used to identify
	local minimizers of \eqref{eq:BPP} where the latter is not necessarily partially calm.
\end{remark}

Below, we want to clarify the relationship between condition \eqref{eq:no_descent_direction}
and the property of $(\bar x,\bar y)$ to be a KKT-point of \eqref{eq:LLVF}.
Therefore, let us state the (nonsmooth) KKT-conditions of \eqref{eq:LLVF} at a reference point $(\bar x,\bar y)\in M$
such that $\varphi$ is locally Lipschitzian at $\bar x$.
We say that $(\bar x,\bar y)$ is a KKT-point of \eqref{eq:LLVF} if there are multipliers
$\nu^G\in\R^p$, $\nu^\textup{vf}\geq 0$, and $\nu^g\in\R^q$ satisfying
	\begin{subequations}\label{eq:KKT_LLVF_abstract}
		\begin{align}
			\label{eq:KKT_LLVF_abstract_x}
				&\nabla_x F(\bar x,\bar y)+\nabla_x G(\bar x,\bar y)^\top\nu^G+\nu^\textup{vf}\nabla_xf(\bar x,\bar y)
					+\nabla_xg(\bar x,\bar y)^\top\nu^g\in\nu^\textup{vf}\partial^c\varphi(\bar x),\\
			\label{eq:KKT_LLVF_abstract_y}	
				&0=\nabla_yF(\bar x,\bar y)+\nabla_yG(\bar x,\bar y)+\nu^\textup{vf}\nabla_yf(\bar x,\bar y)+\nabla_yg(\bar x,\bar y)^\top\nu^g,\\
			\label{eq:KKT_LLVF_abstract_rho}
				&\nu^G\geq 0,\quad\forall i\notin\bar I^G\colon\,\nu^G_i=0,\\
			\label{eq:KKT_LLVF_abstract_mu}
				&\nu^g\geq 0,\quad\forall i\notin\bar I^g\colon\,\nu^g_i=0.
		\end{align}
	\end{subequations}
Let us mention that it is not necessary to state a complementarity slackness condition w.r.t.\ the
optimal value constraint $f(x,y)-\varphi(x)\leq 0$ and its multiplier $\nu^\textup{vf}$ since the constraint
is active at each feasible point of \eqref{eq:BPP}.
Observe that the above definition is reasonable since in case where $\varphi$ is continuously differentiable
at $\bar x$, $\partial^c\varphi(\bar x)$ reduces to the singleton $\{\nabla\varphi(\bar x)\}$ and
\eqref{eq:KKT_LLVF_abstract} coincides with the standard KKT-system of \eqref{eq:LLVF}.
\begin{proposition}\label{prop:B_St_vs_KKT_of_LLVF}
	Let $(\bar x,\bar y)\in M$ be chosen such that $\varphi$ is directionally differentiable and
	locally Lipschitz continuous at $\bar x$.
	Then, the following assertions hold.
	\begin{enumerate}
		\item If \eqref{eq:no_descent_direction} holds and if $\varphi$ is Clarke regular at $\bar x$,
			then $(\bar x,\bar y)$ is a KKT-point of \eqref{eq:LLVF}.
		\item If $(\bar x,\bar y)$ is a KKT-point of \eqref{eq:LLVF} and if $-\varphi$ is Clarke-regular
			at $\bar x$, then \eqref{eq:no_descent_direction} holds.
	\end{enumerate}
\end{proposition}
\begin{proof}
	We show both statements separately.
	\begin{enumerate}
		\item For brevity, we introduce a polyhedral cone $\mathcal K(\bar x,\bar y)\subseteq\R^n\times\R^m$ by
			\[
				\mathcal K(\bar x,\bar y)
				:=
				\left\{
					d\in\R^n\times\R^m\,\middle|\,
						\begin{aligned}
							\nabla G_i(\bar x,\bar y)^\top d&\,\leq\,0&&i\in \bar I^G\\
							\nabla g_i(\bar x,\bar y)^\top d&\,\leq\,0&&i\in \bar I^g
						\end{aligned}
				\right\}.
			\]
			By definition and the assumptions on $\varphi$, we have
			\begin{align*}
				\mathcal L^\varphi_M(\bar x,\bar y)
				&=
				\{d:=(\delta_x,\delta_y)\in\mathcal K(\bar x,\bar y)
				\,|\,
				\nabla f(\bar x,\bar y)^\top d-\varphi'(\bar x;\delta_x)\leq 0
				\}\\
				&=
				\{d:=(\delta_x,\delta_y)\in\mathcal K(\bar x,\bar y)
				\,|\,
				\nabla f(\bar x,\bar y)^\top d-\varphi^\circ(\bar x;\delta_x)\leq 0
				\}\\
				&=
				\{d:=(\delta_x,\delta_y)\in\mathcal K(\bar x,\bar y)
				\,|\,
				\nabla f(\bar x,\bar y)^\top d-\max\{\xi^\top\delta_x\,|\,\xi\in\partial^c\varphi(\bar x)\}\leq 0
				\}\\
				&=
				\bigcup\limits_{\xi\in\partial^c\varphi(\bar x)}
				\{d:=(\delta_x,\delta_y)\in\mathcal K(\bar x,\bar y)
				\,|\,
				\nabla f(\bar x,\bar y)^\top d-\xi^\top\delta_x\leq 0
				\}.
			\end{align*}
			Applying the polarization rule for unions and polyhedral cones yields
			\begin{align*}
				\mathcal L^\varphi_M(\bar x,\bar y)^\circ
				&=
				\bigcap\limits_{\xi\in\partial^c\varphi(\bar x)}
				\{d:=(\delta_x,\delta_y)\in\mathcal K(\bar x,\bar y)
				\,|\,
				\nabla f(\bar x,\bar y)^\top d-\xi^\top\delta_x\leq 0
				\}^\circ\\
				&=
				\bigcap\limits_{\xi\in\partial^c\varphi(\bar x)}
				\left\{
					(\eta,\theta)\,\middle|\,
					\begin{aligned}
						&\eta=\nabla_x G(\bar x,\bar y)^\top\nu^G+\nu^\textup{vf}(\nabla_xf(\bar x,\bar y)-\xi)+\nabla_xg(\bar x,\bar y)^\top\nu^g,\\
						&\theta=\nabla_yG(\bar x,\bar y)^\top\nu^G+\nu^\textup{vf}\nabla_yf(\bar x,\bar y)+\nabla_yg(\bar x,\bar y)^\top\nu^g,\\
						&\nu^G,\nu^\textup{vf},\nu^g\geq 0,\\
						&\forall i\notin\bar I^G\colon\,\nu^G_i=0,\\
						&\forall i\notin\bar I^g\colon\,\nu^g_i=0
					\end{aligned}
				\right\}.
			\end{align*}
			Now, the statement follows observing that \eqref{eq:no_descent_direction} equals
			$-\nabla F(\bar x,\bar y)\in\mathcal L^\varphi_M(\bar x,\bar y)^\circ$ by definition
			of the polar cone.
		\item Let $(\bar x,\bar y)$ be a KKT-point of \eqref{eq:LLVF}. Then, there are multipliers
			$\nu^G\in\R^p$, $\nu^\textup{vf}\geq 0$, and $\nu^g\in\R^q$ as well as $\xi\in\partial^c\varphi(\bar x)$
			which satisfy \eqref{eq:KKT_LLVF_abstract_y}, \eqref{eq:KKT_LLVF_abstract_rho}, \eqref{eq:KKT_LLVF_abstract_mu},
			and
			\begin{equation}\label{eq:KKT_LLVF_abstract_x_special_xi}
				0=\nabla_xF(\bar x,\bar y)+\nabla_xG(\bar x,\bar y)^\top\nu^G+\nu^\textup{vf}(\nabla_xf(\bar x,\bar y)-\xi)
					+\nabla_xg(\bar x,\bar y)^\top\nu^g.
			\end{equation}
			Now, choose $d:=(\delta_x,\delta_y)\in\mathcal L^\varphi_M(\bar x,\bar y)$ arbitrarily. Multiplying
			\eqref{eq:KKT_LLVF_abstract_x_special_xi} with $\delta_x$ as well as \eqref{eq:KKT_LLVF_abstract_y} with $\delta_y$
			and summing the resulting equations up yields
			\begin{align*}
				\nabla F(\bar x,\bar y)^\top d
				&=
				-(\nabla G(\bar x,\bar y)d)^\top\nu^G-\nu^\textup{vf}(\nabla f(\bar x,\bar y)^\top d-\xi^\top\delta_x)
				-(\nabla g(\bar x,\bar y)d)^\top\nu^g\\
				&\geq
				\nu^\textup{vf}(\xi^\top\delta_x-\nabla f(\bar x,\bar y)^\top d)\\
				&\geq
				\nu^\textup{vf}(\min\{\zeta^\top\delta_x\,|\,\zeta\in\partial^c\varphi(\bar x)\}-\nabla f(\bar x,\bar y)^\top d)\\
				&=
				\nu^\textup{vf}(-\max\{(-\zeta)^\top\delta_x\,|\,\zeta\in\partial^c\varphi(\bar x)\}-\nabla f(\bar x,\bar y)^\top d)\\
				&=
				\nu^\textup{vf}(-\max\{\zeta^\top\delta_x\,|\,\zeta\in\partial^c(-\varphi)(\bar x)\}-\nabla f(\bar x,\bar y)^\top d)\\
				&=
				\nu^\textup{vf}(-(-\varphi)^\circ(\bar x;\delta_x)-\nabla f(\bar x,\bar y)^\top d)\\
				&=
				\nu^\textup{vf}(-(-\varphi)'(\bar x;\delta_x)-\nabla f(\bar x,\bar y)^\top d)\\
				&=
				\nu^\textup{vf}(\varphi'(\bar x;\delta_x)-\nabla f(\bar x,\bar y)^\top d)
				\geq
				0,			
			\end{align*}
			and this shows the desired result.
	\end{enumerate}
	Hence, the proof is complete.
\end{proof}

We close the abstract analysis in this section with a brief remark regarding the above result.
\begin{remark}\label{rem:B_St_LLVF_varphi_Clarke_regular}
	The proof of \cref{prop:B_St_vs_KKT_of_LLVF} even shows the following stronger statement:
	Fix $(\bar x,\bar y)\in M$ such that $\varphi$ is directionally differentiable,
	locally Lipschitz continuous, and Clarke regular at $\bar x$ (all these assumptions hold if $\varphi$
	is locally convex around $\bar x$). If \eqref{eq:no_descent_direction}
	holds, then, for each $\xi\in\partial^c\varphi(\bar x)$, there exist multipliers
	$\nu^G\in\R^p$, $\nu^\textup{vf}\geq 0$, and $\nu^g\in\R^q$ which satisfy
	\eqref{eq:KKT_LLVF_abstract_x_special_xi}, \eqref{eq:KKT_LLVF_abstract_y}, \eqref{eq:KKT_LLVF_abstract_rho},
	and \eqref{eq:KKT_LLVF_abstract_mu}.
\end{remark}

\subsection{Fully linear lower level}\label{sec:fully_linear_lower_level}

Here, we focus on the setting
\begin{equation}\label{eq:fully_linear_lower_level}
	\forall x\in\R^n\colon\quad
	S(x):=\argmin\limits_y\{c^\top y\,|\,Ax+By\leq b\}
\end{equation}
where $A\in\R^{q\times n}$, $B\in\R^{q\times m}$, $b\in\R^q$, and $c\in\R^m$ are fixed matrices.
Obviously, this particular lower level problem is fully linear and, thus, the associated
optimal value function $\varphi$ is convex and piecewise affine on $\dom\varphi$. More precisely, there
exist only finitely many so-called regions of stability such that $\varphi$ is affine on each
of these sets. In particular, $\varphi$ can be represented as the maximum of finitely many
affine functions on $\dom\varphi$. As a consequence, the maximum-rules from
\cref{lem:max_rule,lem:second_order_epiregularity_max} show that $\varphi$ is second-order
directionally differentiable at each point $x\in\intr\dom\varphi$.
Clearly, $\varphi$ is Lipschitz continuous on $\intr\dom\varphi$.

Next, we fix a point $\bar x\in\dom\varphi$. Then, $S(\bar x)\neq\varnothing$ holds true, and
by strong duality, the solution set of the lower level dual, given by
\[
	\bar\Lambda(\bar x):=\argmax\limits_\lambda\{(A\bar x-b)^\top\lambda\,|\,B^\top\lambda=-c,\,\lambda\geq 0\},
\]
is nonempty as well. With the aid of this set, we can express the subdifferential of $\varphi$ at $\bar x$.
\begin{lemma}\label{lem:subdifferential_fully_linear_lower_level}
Fix $\bar x\in\dom\varphi$. Then, it holds
\[
	\partial\varphi(\bar x)=\{A^\top\lambda\,|\,\lambda\in \bar\Lambda(\bar x)\}.
\]
\end{lemma}
\begin{proof}
	We show both inclusions separately. 
	For the first one, fix $\xi\in\partial\varphi(\bar x)$ and some $\bar y\in S(\bar x)$.
	Then, it holds $c^\top\bar y=\varphi(\bar x)$.
	By definition of the subdifferential (in the sense of convex analysis), $(\bar x,\bar y)$ solves
	the linear programming problem
	\[
		\min\limits_{x,y}\{-\xi^\top x+c^\top y\,|\,Ax+By\leq b\},
	\]
	i.e., we can find some $\lambda\in\R^q$ such that
	\[
		-\xi+A^\top\lambda=0,\quad c+B^\top\lambda=0,\quad\lambda^\top(A\bar x+B\bar y-b)=0, \quad \lambda\geq 0.
	\]
	By strong duality of linear programming, the last three conditions ensure $\lambda\in \bar\Lambda(\bar x)$.

	Conversely, fix $\bar\lambda\in \bar\Lambda(\bar x)$ and assume that $A^\top\bar\lambda$ does
	not belong to $\partial\varphi(\bar x)$. Then, we find some $x\in\R^n$ such that
	$\varphi(x)<\varphi(\bar x)+\bar\lambda^\top A(x-\bar x)$ holds true. By definition of $\varphi$, there need to
	exist $y\in K(x)$ and $\bar y\in S(\bar x)$ which satisfy $c^\top y<c^\top\bar y+\bar\lambda^\top A(x-\bar x)$.
	Due to $\bar\lambda\in \bar\Lambda(\bar x)$, it holds
	\[
		c+B^\top\bar\lambda=0,\quad\bar\lambda^\top(A\bar x+B\bar y-b)=0,\quad \bar\lambda\geq 0,
	\]
	and this yields
	\begin{align*}
		c^\top y
		<
		c^\top\bar y+\bar\lambda^\top A(x-\bar x)
		\leq
		c^\top\bar y+\bar\lambda^\top (b-By)+\bar\lambda^\top(B\bar y-b)
		=
		(c+B^\top\bar\lambda)^\top\bar y-(B^\top\bar\lambda)^\top y
		=c^\top y,
	\end{align*}
	which is a contradiction.
\end{proof}

Now, we assume that $\bar x\in\intr\dom\varphi$. This guarantees that $\partial\varphi(\bar x)$ is compact.
Thus, \cref{lem:subdifferential_fully_linear_lower_level} yields
\[
	\forall \delta_x\in\R^n\colon\quad
	\varphi'(\bar x;\delta_x)=\max\limits_{\lambda\in \bar\Lambda(\bar x)} \delta_x^\top A^\top\lambda.
\]
Noting that locally around $\bar x$, $\varphi$ equals the maximum of finitely many affine functions, we have
\[
	\varphi(\bar x+tr)-\varphi(\bar x)=t\varphi'(\bar x;r)
\]
for all $r\in\R^n$ provided $t>0$ is sufficiently small. Particularly, we obtain
\[
	\varphi(\bar x+t\delta_x+\tfrac12t^2\omega_x)-\varphi(\bar x)
	=t\max\limits_{\lambda\in \bar\Lambda(\bar x)}\{\delta_x^\top A^\top\lambda +\tfrac12t\omega_x^\top A^\top\lambda\}
\]
for all $\delta_x,\omega_x\in\R^n$ if only $t>0$ is sufficiently small.
Noting that $A^\top \bar\Lambda(\bar x)$ is a compact polyhedron, $\delta_x^\top A^\top\lambda$ dominates the maximization term as
$t>0$ is small. As a consequence, for sufficiently small $t>0$, we have
\begin{equation}\label{eq:Taylor_like_estimate_fully_linear_lower_level_value_function}
	\varphi(\bar x+t\delta_x+\tfrac12t^2\omega_x)-\varphi(\bar x)
	=
	t\varphi'(\bar x;\delta_x)
	+
	\tfrac12t^2\max\limits_{\bar \lambda}
		\left\{
			\omega_x^\top A^\top\bar\lambda\,\middle|
				\,\bar \lambda\in\argmax\limits_{\lambda\in\bar\Lambda(\bar x)} \delta_x^\top A^\top\lambda
		\right\}.
\end{equation}
As a result of these arguments, it follows
\[
	\forall \delta_x,\omega_x\in\R^n\colon\quad
	\varphi''(\bar x;\delta_x,\omega_x)
	=
	\max\limits_{\bar\lambda}
		\left\{
			\omega_x^\top A^\top\bar\lambda\,\middle|\,\bar\lambda\in\argmax\limits_{\lambda\in \bar\Lambda(\bar x)} \delta_x^\top A^\top\lambda
		\right\},
\]
and due to \eqref{eq:Taylor_like_estimate_fully_linear_lower_level_value_function}, $\varphi$ and $-\varphi$
both are second-order epi-regular at $\bar x$.
The above formula for the second-order directional derivative of $\varphi$
matches available results on the second-order directional derivative of optimal value functions
in nonlinear parametric optimization, see, e.g., \cite[Theorem~4.1]{Shapiro1988} and \cite[Theorem~4.2]{Shapiro1988b}.
The full linearity of the lower level problem causes $\varphi$ to be a curvature-free function possessing some
kinks. In the above formula, this is reflected by the fact that no second-order information on the lower level
data appears since all second-order derivatives vanish.

Let us note that local minimizers of \eqref{eq:BPP} where the lower level stage is given as in \eqref{eq:fully_linear_lower_level}
while $G$ is affine are KKT-points of \eqref{eq:LLVF}, see \cite[Corollary~4.1]{Ye2004},
and this statement is independent of the underlying upper level objective function.
Even in the case where $G$ is an arbitrary smooth map, the inherent partial calmness at all local minimizers, see
\cite[Proposition~4.1]{YeZhu1995}, implies that under
validity of reasonable constraint qualifications, the KKT-conditions of \eqref{eq:LLVF} provide a necessary optimality
condition for \eqref{eq:BPP}.
As a consequence, there is some reasonable hope that the first-order stationarity condition \eqref{eq:no_descent_direction}
holds in the present setting although this is not naturally guaranteed if $\varphi$ is nonsmooth at the reference point,
see \cref{prop:B_St_vs_KKT_of_LLVF}.
The above formula for the directional derivative of $\varphi$ can be used to characterize the underlying linearization cone
and critical cone explicitly. Furthermore, \eqref{eq:second_order_sufficient_condition_subproblem} can be characterized in
terms of initial problem data. This allows us to evaluate the second-order sufficient optimality condition from
\cref{thm:abstract_second_order_condition}.

\subsection{Linear lower level with parameters in the objective function}\label{sec:linear_lower_level_param_in_obj}

Here, we focus on the lower level problem given by
\begin{equation*}
	\forall x\in\R^n\colon\quad
	S(x):=\argmin\limits_y\{(Ax+c)^\top y\,|\,By\leq b\}
\end{equation*}
where $A\in\R^{m\times n}$, $B\in\R^{q\times m}$, $b\in\R^q$, and $c\in\R^m$ are fixed matrices.
For simplicity, let us assume that the set $K:=\{y\in\R^m\,|\,By\leq b\}$ is compact.
Then, $\dom S=\R^n$ holds true and the associated optimal value function $\varphi\colon\R^n\to\R$
is given by
\[
	\forall x\in\R^n\colon\quad
	\varphi(x)=\min\{(Ax+c)^\top y_l\,|\,l=1,\ldots,\ell\}
\]
where $y_1,\ldots,y_\ell\in K$ are the (finitely many) vertices of $K$.
Observing that $\varphi$ is the minimum of finitely many affine functions, it is globally Lipschitz
continuous. Using a minimum-rule similar to the maximum-rule from \cref{lem:max_rule}, one can
easily check that $\varphi$ is directionally differentiable, and it holds
\[
	\forall\delta_x\in\R^n\colon\quad
	\varphi'(\bar x;\delta_x)=\min\{\delta_x^\top A^\top y_l\,|\,l\in I(\bar x)\}
\]
for all $\bar x\in\R^n$ where we used
\[
	I(\bar x):=\{\bar l\in\{1,\ldots,\ell\}\,|\,(A\bar x+c)^\top y_{\bar l}
	=\min\{(A\bar x+c)^\top y_l\,|\,l=1,\ldots,\ell\}\}.
\]
Noting that we have $S(\bar x)=\conv\{y_l\,|\,l\in I(\bar x)\}$, it follows
\[
	\forall\delta_x\in\R^n\colon\quad
	\varphi'(\bar x;\delta_x)=\min\{\delta_x^\top A^\top y\,|\,y\in S(\bar x)\}.
\]
Similar arguments as used in \cref{sec:fully_linear_lower_level} can be exploited to infer that $\varphi$ is
second-order directionally differentiable and that the formula
\[
	\forall \delta_x,\omega_x\in\R^n\colon\quad
	\varphi''(\bar x;\delta_x,\omega_x)
	=
	\min\limits_{\tilde{y}}\left\{
		\omega_x^\top A^\top \tilde y\,\middle|\,
		\tilde y\in\argmin\limits_{y\in S(\bar x)} \delta_x^\top A^\top y
	\right\}
\]
holds for all $\bar x\in\R^n$.
Trivially, $-\varphi$ is second-order directionally differentiable, too.
Noting that $\epi(-\varphi)$ is a polyhedron, $-\varphi$ is second-order epi-regular,
see \cref{sec:generalized_differentiation}.
\begin{example}\label{ex:second_order_sufficient_condition_linear_lower_level_parameter_in_objective}
	Let us consider the bilevel programming problem
	\[
		\min\limits_{x,y}\{\tfrac12(x+3)^2+\tfrac12y^2\,|\,y\in S(x)\}
	\]
	where $S\colon\R\rightrightarrows\R$ is the solution map of $\min_y\{xy\,|\,y\in[0,1]\}$.
	Observe that this is a slightly modified version of the problem discussed in
	\cref{ex:first_order_sufficient_condition,ex:second_order_sufficient_condition_varphi_locally_smooth}.
	With the aid of \cref{thm:abstract_second_order_condition}, we investigate the point
	$(\bar x,\bar y):=(0,0)$ and note that $\varphi$ is nonsmooth there.
	However, the above considerations show that $\varphi$ is Lipschitzian and second-order directionally
	differentiable while $-\varphi$ is second-order epi-regular.
	
	Obviously, we have $\varphi'(\bar x;\delta_x)=\min\{\delta_x,0\}$ for all $\delta_x\in\R$.
	This leads to
	\[
		\mathcal L^\varphi_M(\bar x,\bar y)=\R^2_+,\qquad
		\mathcal C^\varphi_M(\bar x,\bar y)=\{0\}\times\R_+.
	\]
	Each nonvanishing critical direction from $\mathcal C^\varphi_M(\bar x,\bar y)$ is of the form
	$d:=(0,\delta_y)$ with $\delta_y>0$. Hence, in \eqref{eq:second_order_sufficient_condition_subproblem},
	$\varphi''(\bar x;0,\omega_x)$ needs to be evaluated, and the above formula gives us
	$\varphi''(\bar x;0,\omega_x)=\min\{\omega_x,0\}$ for all $\omega_x\in \R$.
	Furthermore, we have $\bar I^g(d)=\varnothing$ in the situation at hand.
	Problem \eqref{eq:second_order_sufficient_condition_subproblem} therefore reduces to
	\[
		\min\limits_{\omega_x,\omega_y,\alpha}
		\{
			\alpha
			\,|\,
			3\omega_x+\delta_y^2\leq\alpha,\,-\min\{\omega_x,0\}\leq\alpha
		\}.
	\]
	For $\omega_x\geq 0$, the first constraint implies $\alpha\geq\delta_y^2$.
	If $\omega_x\in(-\tfrac{1}{4}\delta^2_y,0)$ holds, then the first constraint yields $\alpha\geq \tfrac14\delta^2_y$.
	For $\omega_x\leq-\tfrac14\delta^2_y$, we infer $\alpha\geq -\omega_x\geq\tfrac14\delta^2_y$ from the second
	constraint. Thus, for each
	$d:=(0,\delta_y)\in\mathcal C^\varphi_M(\bar x,\bar y)\setminus\{0\}$, the optimal value of the associated problem
	\eqref{eq:second_order_sufficient_condition_subproblem} is not smaller than $\tfrac14\delta^2_y$.
	Consequently, \cref{thm:abstract_second_order_condition} shows that $(\bar x,\bar y)$ is a strict local
	minimizer of the given bilevel programming problem.
	
	Let us briefly note that the first-order sufficient optimality condition from \cref{thm:first_order_sufficient_condition}
	does \emph{not} hold at $(\bar x,\bar y)$. This is due to the fact that the first-order derivative w.r.t.\ $y$ of the
	upper level objective function vanishes at the reference point, cf. \cref{ex:first_order_sufficient_condition} where this
	does not happen.
\end{example}

\subsection{Stable unique lower level solutions}\label{sec:unique_lower_level_solution_under_LLICQ}

In this section, we take a closer look at situations where the lower level solution under consideration
is unique which is provided via a standard second-order sufficient condition and convexity. Under validity
of LLICQ, one does not only obtain the continuous differentiability of $\varphi$  but also its
second-order directional differentiability as well as a ready-to-use formula for the computation of the
second-order directional derivative. These results are subsumed in the upcoming proposition.
\begin{proposition}\label{prop:locally_unique_lower_level_solution}
	For each $x\in\R^n$, let $f(x,\cdot)\colon\R^m\to\R$ be convex and let $g(x,\cdot)\colon\R^m\to\R^q$ be
	componentwise convex.
	Fix a point $(\bar x,\bar y)\in M$ where LLICQ and LSOSC hold.
	Let $\bar\lambda\in\R^q$ be the uniquely determined associated lower level Lagrange multiplier.
	Then, the following assertions hold:
	\begin{enumerate}
		\item $\varphi$ is continuously differentiable at $\bar x$ and it holds
			$\nabla\varphi(\bar x)=\nabla_xL(\bar x,\bar y,\bar\lambda)$,
		\item $\varphi$ is second-order directionally differentiable at $\bar x$ and it holds
			\begin{equation}\label{eq:second_order_dir_der_varphi_unique_stable_lower_level_solution}
				\begin{split}
				\varphi''(\bar x;\delta_x,\omega_x)
				&=
				\nabla_xL(\bar x,\bar y,\bar\lambda)\omega_x\\
				&\quad +\inf\limits_{\delta_y}
				\left\{
					d^\top\nabla^2L(\bar x,\bar y,\bar\lambda)d
					\,\middle|\,
					\begin{aligned}
						\nabla g_i(\bar x,\bar y)^\top d &\,=\,0 &&i\in\bar I^g,\,\bar\lambda_i>0\\
						\nabla g_i(\bar x,\bar y)^\top d &\,\leq\, 0 &&i\in \bar I^g,\,\bar\lambda_i=0
					\end{aligned}
				\right\}
				\end{split}
			\end{equation}
			for all $\delta_x,\omega_x\in\R^n$ where we used $d:=(\delta_x,\delta_y)\in\R^n\times\R^m$ again,
		\item $\varphi$ and $-\varphi$ are second-order epi-regular at $\bar x$.
	\end{enumerate}
\end{proposition}
\begin{proof}
	Under the assumptions made, there are a neighborhood $U\subseteq\R^n$ of $\bar x$,
	a mapping $s\colon U\to\R^m$, and a constant $\kappa>0$ such that $S(x)=\{s(x)\}$ holds
	for all $x\in U$ while the upper Lipschitz property
	\[
		\forall x\in U\colon\quad
		\norm{s(x)-\bar y}{2}\leq \kappa\norm{x-\bar x}{2}
	\]
	is valid, see, e.g., \cite[Section~4]{Robinson1982}.
	Noting that validity of LLICQ implies that the Strict Mangasarian--Fromovitz Condition holds
	for the lower level problem at $(\bar x,\bar y)$, the assertions of the proposition directly
	follow by applying \cite[Theorem~4.139]{BonnansShapiro2000} as well as the subsequently stated
	remarks.
\end{proof}

\begin{remark}\label{rem:locally_unique_lower_level_solutions_linearization_cone}
	Suppose that all the assumptions of \cref{prop:locally_unique_lower_level_solution} are valid.
	Then, it holds
	\[
		\mathcal L^\varphi_M(\bar x,\bar y)
		=
		\left\{
		d:=(\delta_x,\delta_y)\in\R^n\times\R^m\,\middle|\,
			\begin{aligned}
				\nabla G_i(\bar x,\bar y)^\top d&\,\leq\,0&&i\in \bar I^G\\
				\nabla f(\bar x,\bar y)^\top d-\nabla_xL(\bar x,\bar y,\bar\lambda)^\top\delta_x&\,\leq\,0&&\\
				\nabla g_i(\bar x,\bar y)^\top d&\,\leq\,0&&i\in \bar I^g
			\end{aligned}
		\right\}.
	\]
	Since $\varphi$ is continuously differentiable, \cref{prop:B_St_vs_KKT_of_LLVF} yields that
	\eqref{eq:no_descent_direction} is equivalent to the existence of multipliers
	$\nu^G\in\R^p$, $\nu^\textup{vf}\geq 0$, and $\nu^{g}\in\R^q$ which satisfy
	\eqref{eq:KKT_LLVF_abstract_y}, \eqref{eq:KKT_LLVF_abstract_rho}, \eqref{eq:KKT_LLVF_abstract_mu}, and
	\[
		0=\nabla_x F(\bar x,\bar y)+\nabla_x G(\bar x,\bar y)^\top\nu^G+\nabla_xg(\bar x,\bar y)^\top(\nu^g-\nu^\textup{vf}\bar\lambda).
	\]
\end{remark}

Let us finish this section with another illustrative example.
\begin{example}\label{ex:unique_lower_level_solution_varphi_not_twice_differentiable}
	Let us consider the bilevel programming problem
	\[
		\min\limits_{x,y}\{(x-8)^2+(y-9)^2\,|\,x\geq 0,\,y\in S(x)\}
	\]
	where $S\colon\R\rightrightarrows\R$ is the solution map of $\min_y\{(y-3)^2\,|\,y^2\leq x\}$.
	This lower level problem is taken from \cite{Dempe2002}.
	One can easily check that
	\[
		\forall x\in\R\colon\quad
		S(x)=
			\begin{cases}
              \varnothing & \mbox{ if }  x<0,\\
               \{\sqrt{x}\} & \mbox{ if }  0\leq x \leq 9,\\
                \{3\} & \mbox{ if }  x>9,
            \end{cases}
       \qquad \mbox{ and }\qquad
       \varphi(x)=
       		\begin{cases}
              \infty & \mbox{ if }  x<0,\\
               (\sqrt{x} - 3)^2 & \mbox{ if }  0\leq x \leq 9,\\
                0 & \mbox{ if }  x>9.
            \end{cases}
     \]
	Consider the point $(\bar x,\bar y):=(9,3)$.
	For $\bar x$, the associated lower level problem satisfies
	LLICQ and LSOSC at $\bar y$. The uniquely determined
	associated Lagrange multiplier is given by $\bar\lambda:=0$.
	One obtains
     \[
     	\mathcal L^\varphi_M(\bar x,\bar y)=\{(\delta_x,\delta_y)\in\R^2\,|\,6\delta_y-\delta_x\leq 0\},
     \]
     and this leads to validity of \eqref{eq:no_descent_direction}. Furthermore, it holds
     \[
     	\mathcal C^\varphi_M(\bar x,\bar y)=\{(\delta_x,\delta_y)\in\R^2\,|\,6\delta_y-\delta_x= 0\}.
     \]
     Using formula \eqref{eq:second_order_dir_der_varphi_unique_stable_lower_level_solution}, one obtains
     \[
     	\forall\delta_x,\omega_x\in\R\colon\quad
     	\varphi''(\bar x;\delta_x,\omega_x)
     	=
     	\begin{cases}
     		0	&\text{if }\delta_x\geq 0\\
     		\tfrac1{18}\delta_x^2	&\text{if }\delta_x<0.
     	\end{cases}
     \]
     Let $(\delta_x,\delta_y)\in\mathcal C^\varphi_M(\bar x,\bar y)$ be a nonvanishing direction.
     Clearly, it holds $\delta_y=\tfrac16\delta_x$. Thus, \eqref{eq:second_order_sufficient_condition_subproblem}
     simplifies to
     \[
     	\min\limits_{\omega_x,\omega_y,\alpha}
     	\left\{
     		\alpha
     		\,\middle|\,
     		2\omega_x-12\omega_y+(2+\tfrac1{18})\delta_x^2\leq\alpha,\,
     		\tfrac{1}{18}\delta_x^2-\varphi''(\bar x;\delta_x,\omega_x)\leq\alpha,
     		-\omega_x+6\omega_y+\tfrac1{18}\delta_x^2\leq\alpha
     	\right\}.
     \]
     In case $\delta_x>0$, the second constraint yields that the minimal value of this program is bounded from below
     by $\tfrac1{18}\delta_x^2$. Supposing that $\delta_x<0$ holds, we multiply the third constraint by $2$ and add
     it to the first one in order to see that the minimal objective value is bounded from below by
     $\tfrac{13}{18}\delta_x^2$. Consequently, \cref{thm:abstract_second_order_condition} shows that $(\bar x,\bar y)$
     is a strict local minimizer of the given bilevel optimization problem.
\end{example}

\section{Conclusion and perspectives}\label{sec:conclusions}

In this paper, we derived first- and second-order sufficient optimality conditions for optimistic bilevel
optimization problems of the form \eqref{eq:BPP}.
For the first-order sufficient conditions, we exploited reasonable upper estimates of
the tangent cone to the bilevel feasible set. Second-order sufficient conditions were obtained using the
value function reformulation \eqref{eq:LLVF} of \eqref{eq:BPP} as well as an appropriate second-order
directional derivative. 
Some relations to the popular partial calmness property, which is equivalent
to local exactness of the optimal value constraint as a penalty function, were discussed.
In light of \cite{HanMangasarian1979}, it has to be clarified in the future which types of (strong) second-order
sufficient optimality conditions addressing \eqref{eq:BPP} already
imply partial calmness at the underlying strict local minimizers, since such second-order conditions might
be too selective in light of the rare number of situations where partial calmness is inherent.
Observing that second-order sufficient optimality conditions are available for mathematical programs with
complementarity constraints, see \cite[Theorem~7]{ScheelScholtes2000}, one could investigate how these
conditions apply to the KKT-reformulation \eqref{eq:KKT_reformulation} of \eqref{eq:BPP}. Keeping in mind
that the second-order growth condition cannot hold for \eqref{eq:KKT_reformulation} whenever there exist
multiple lower level Lagrange multipliers at the reference point, see \cref{rem:non_unique_multipliers},
it has to be studied whether this theory can be generalized to the less restrictive situation where the
mapping $\Lambda$ is allowed to be multi-valued.



\end{document}